\documentclass[11pt]{amsart}

\allowbreak

\numberwithin{equation}{section}

\usepackage{amsfonts,amssymb,amscd,amsmath,latexsym,amsbsy}
\usepackage{graphicx}
\usepackage{subcaption}
\usepackage{amssymb}
\usepackage{amsfonts}
\usepackage{latexsym}
\usepackage{mathrsfs}
\usepackage{mathtools}
\usepackage{amsthm}
\usepackage{comment}
\usepackage{listings}
\usepackage{tikz-cd}
\usepackage{tikz}
\usetikzlibrary {calc}
\usepackage{hyperref}
\usepackage{amsmath}
\usepackage{leftidx}
\usepackage[toc,page]{appendix}
\usepackage{upgreek}

\allowdisplaybreaks
\newtheorem{thm}{Theorem}[section]

\newtheorem{lmm}[thm]{Lemma}
\newtheorem{prop}[thm]{Proposition}

\theoremstyle{definition}
\newtheorem{definition}[thm]{Definition}
\newtheorem{remark}[thm]{Remark}
\newtheorem*{remark*}{Remark}

\numberwithin{equation}{section}

\newcommand{\SM}{\mathcal M}
\newcommand{\SR}{\mathcal R}

\newcommand{\R}{\ensuremath{\mathbb{R}}}
\newcommand{\RR}{\ensuremath{\mathbb{R}}}
\newcommand{\Z}{\ensuremath{\mathbb{Z}}}
\newcommand{\N}{\ensuremath{\mathbb{N}}}

\newcommand{\WF}{\ensuremath{\mathrm{WF}}}
\newcommand{\Ell}{\ensuremath{\mathrm{Ell}}}
\newcommand{\supp}{\ensuremath{\mathrm{supp}}}

\newcommand{\mk}{\ensuremath{\mathfrak}}

\newcommand{\msf}{\ensuremath{\mathsf}}

\newcommand{\la}{\ensuremath{\langle}}
\newcommand{\ra}{\ensuremath{\rangle}}

\newcommand{\Op}{\ensuremath{\mathrm{Op}}}

\newcommand{\Id}{\ensuremath{\mathrm{Id}}}
\newcommand{\oc}{\ensuremath{\mathrm{1c}}}

\newcommand{\ps}{\ensuremath{\mathrm{ps}}}
\newcommand{\full}{\ensuremath{\mathrm{full}}}
\newcommand{\sub}{\ensuremath{\mathrm{sub}}}

\newcommand{\ff}{\ensuremath{\mathrm{ff}}}
\newcommand\SL{\mathcal{L}}

\newcommand\xoc{{x_{\oc}}}
\newcommand\yoc{{y_{\oc}}}
\newcommand\etaoc{{\eta_{\oc}}}
\newcommand\xioc{{\xi_{\oc}}}

\newcommand\rhops{{\rho_{\ps}}}
\newcommand\ffocps{{\ff_{\oc-\ps}}}
 
\newcommand\ffococ{{\ff_{\oc-\oc}}} 

\newcommand\xoco{x_{\oc, 1}}
\newcommand\yoco{y_{\oc, 1}}
\newcommand\xioco{\xi_{\oc, 1}}
\newcommand\etaoco{\eta_{\oc, 1}}
\newcommand\xoct{x_{\oc, 2}}
\newcommand\yoct{y_{\oc, 2}}
\newcommand\xioct{\xi_{\oc, 2}}
\newcommand\etaoct{\eta_{\oc, 2}}

\newcommand{\cl}{\ensuremath{\mathrm{cl}}}

\newcommand{\Char}{\ensuremath{\mathrm{Char}}}

\newcommand{\Cl}{\ensuremath{\mathfrak{Cl}_g}}

\newcommand{\ococb}{\ensuremath{ \leftidx{^{ \mathsf{L}\oc}}{ T^*X^2_b} } }

\newcommand{\rHp}{\ensuremath{ \mathsf{H}^{m',0}_p } }

\newcommand{\Rp}{\mathcal{R}_+}
\newcommand{\Rm}{\mathcal{R}_-}

\newcommand\Legps{L}
\newcommand\Leg{\mathcal{L}}
\newcommand\Lagps{\Lambda}
\newcommand\Lag{\Uplambda}

\newcommand\FSR{\overline{\Lambda_-'}}
\newcommand\BSR{\overline{\Lambda_+'}}

\newcommand\Poi{\mathcal{P}}
\newcommand\Poim{\mathcal{P}_-}
\newcommand\Poip{\mathcal{P}_+}

\newcommand\Poipm{\mathcal{P}_\pm}
\newcommand\Radm{\mathcal{R}_-}
\newcommand\Radp{\mathcal{R}_+}

\newcommand\base{\mathrm{base}}

\newcommand{\ococparaone}{v}
\newcommand{\ococparatwo}{w}

\newcommand\ocphase{\overline{{}^{\oc} T^* \R^n}}
\newcommand\psphase{\overline{{}^{\ps} T^* \R^{n+1}}}

\usepackage[letterpaper,top=2cm,bottom=2cm,left=2cm,right=2cm,marginparwidth=1.75cm]{geometry}

\title[Determining potentials from the scattering map]{Determining potentials from the scattering map of the time-dependent Schr\"odinger equation}

\author{Qiuye Jia}
\date{\today}

\begin{document}

\begin{abstract}
For a time dependent Schr\"odinger equation, the scattering map is the map sending the asymptotic profile of a solution as $t \to-\infty$ to its asymptotic profile as $t\to+\infty$. 
In this paper we show that, for a certain class of metrics, the scattering maps and Poisson operators associated to two Schr\"odinger operators on the same curved space only differ by a compact operator on a critical level if and only if the two potentials are equal.
\end{abstract}

\maketitle

\tableofcontents
\section{Introduction}
\label{sec:intro}

\subsection{The set up and the main results}
\label{subsec:main-results}

In this article, we study the problem of how the scattering map and the Poisson operator of the time-dependent Schr\"odinger equation in curved spaces determine the potential. 
More precisely, we consider the Schr\"odinger operator $P$ on $\R^{n+1}_{z, t}$, where $z \in \R^n$ with $n \geq 3$, $t \in \R$, 
\begin{equation} \label{eq:def-Schrodinger-op}
P = D_t + \Delta_{g(t)} + V(z, t), 
\end{equation}
where $D_t = -i \partial_t$, $\Delta_{g(t)}$ is the positive Laplace operator with respect to a smooth family of metrics $g(t)$ on $\R^n_z$, and $V$ is a smooth real-valued potential function 
\footnote{ The requirement that $V$ is real only enters towards the end of the article in Section~\ref{sec:sc-map-determine-potential}. We allow complex-valued potentials for results recalled from \cite{HJ2026-scattering-map}. Also, we allow complex valued potentials for results in Section~\ref{sec:Poisson-determine-potential} determining the potential from the Poisson operator. Even for results in Section~\ref{sec:sc-map-determine-potential}, we can still determine the real part of the potential when they are complex valued. See Remark~\ref{remark: complex-potential}. }. 
In addition, we make the assumption that $g(t)$ is a compactly supported, in spacetime, perturbation of the flat metric:
\begin{equation} \label{eq:g-definition-perturbation}
g(t) = g_0+ \tilde{g}(t), \text{ where } \tilde{g}(t) \text{ is compactly supported in spacetime}, 
\end{equation}
and $V$ is also compactly supported in spacetime. 
Thus, there exist large constants $R$, $T$ such that if either $|z| \geq R$, or if $|t| \geq T$, then $g(t) = \sum_{i,j} g_{ij}(z,t) dz_i dz_j$ coincides with the flat metric $g_0 = \sum_i dz_i^2$, and $V$ vanishes identically. 
We use $K \subset \R^{n+1}$ to denote a compact set containing the support of $\tilde{g}$ and $V$.
So outside $K$, $P$ coincides with the `free' Schr\"odinger operator $P_0 := D_t + \Delta_0$, where $\Delta_0$ is the (positive) Laplacian on $\R^n$ with the Euclidean metric. We assume, in addition, that each $g(t)$ is non-trapping in the sense that its geodesics always escape any compact region in $\R^n$ in finite time.

The scattering map concerns final state data or scattering data of global solutions to $Pu = 0$. Every  global solution $u$ has an asymptotic expansion of the form 
\begin{equation}\label{eq:u expansion}
u \sim (4\pi it)^{-n/2} e^{i|z|^2/4t} f_\pm\big( \frac{z}{2t} \big) + O(|t|^{-n/2 - \epsilon}), \quad t \to \pm \infty, \, \epsilon>0
\end{equation}
for large positive or negative times, where this holds in a pointwise sense if $f_\pm$ are sufficiently regular, and distributionally in general. The functions $f_\pm$ are by definition the asymptotic data of the solution $u$. 
Then the \emph{scattering map} of $P$ is the map
\begin{equation} \label{eq:S-def-intro}
    S: \; f_- \to f_+.
\end{equation}
The main theorem of \cite{HJ2026-scattering-map} shows that this map is an elliptic Fourier integral operator—more precisely, a Legendre distribution—of a novel type.
See \cite[Section~1]{HJ2026-scattering-map} for more detailed discussion.

The aim of this article is to answer the inverse problem that arises from this. That is, suppose the scattering maps $S_{i}$ associated with two different operators
\begin{equation} \label{eq:def-Pi}
    P_i=D_t+\Delta_{g(t)}+V_i(z,t),\; i=1,2
\end{equation}
coincide up to a certain order, then does this imply $V_1=V_2$? 

Another object of fundamental importance in the scattering theory is the Poisson operator. In terms of the notation above, forward (for $-$) and backward Poisson operators are the maps
\begin{align}
\Poipm: f_\pm \to u.
\end{align}
One can also ask whether $\Poipm$ associated with $P_i$ coinciding up to a certain order implies that $V_1=V_2$.

We give an affirmative answer to both questions when the metric $g(t)$ is in the following class of metrics, which arises naturally in the geometric inverse problems. 

\begin{definition} \label{definition:metric-classes}
Let $g$ be as above and $B_R(0) \subset \R^n$ be the ball in $\R^n$ of radius $R$ centered at the origin and $T \in \R_+$, such that $[-T,T] \times B_R(0)$ contains the support of $g-g_0$. We say that $g(t)$ \emph{admits a convex function} if there is a function $f \in C^\infty([-T,T] \times B_R(0))$ such that $\mathrm{Hess} \, f$ is strictly positive, where the Hessian is with respect to $g(t)$ for fixed $t$ on $B_R(0)$.
\end{definition}

See \cite[Lemma~2.1]{paternain2019geodesic} for conditions ensuring that $g(t)$ admits a convex function. In particular, when $g(t)$ is Euclidean, or has non-positive sectional curvature, or has curvature lower-bounded by a constant depending on $R$, then $g(t)$ admits a convex function. 



For the first question concerning $S_1-S_2$, our answer is the following theorem. 

\begin{thm} \label{thm:main-scmap-intro} 
For $n \geq 3$, suppose $S_1-S_2$ is a compact operator from $L^2(\R^n)$ to $H^{s,1}_{\oc}(\R^n)$ for a fixed $s \in \R$, then $V_1=V_2$.
\end{thm}

\begin{remark}
The condition that $g$ admits a convex function and the dimensional restriction $n \geq 3$ only enters in the last step invoking Theorem~\ref{thm:X-ray}. So one can replace this condition by any other conditions that imply the injectivity of the weighted X-ray transform. 

In addition, combining with the proof of \cite[Corollary]{uhlmann2016inverse},
for fixed $t$ and large $N$ depending on how many iterations we need to foliate the entire $B_R(0)$, 
our proof in Section~\ref{sec:sc-map-determine-potential} gives a global stability estimate of the form 
\begin{equation}
\| V_1 - V_2\|_{H^s(\R^{n})} \lesssim \|\sigma^{-1}_{\oc-\oc}( (S_1-S_2)Q_{\oc} ) \|_{H^{s+N}( \ocphase )}.
\end{equation}
Here we used the fact that points in $\ocphase$ can be identified with geodesics in $\R^{n+1}$. See \cite[Section~3.3]{HJ2026-scattering-map}\cite[Section~5.1]{jia2026metric-inverse} for details.
\end{remark}

\begin{remark} \label{remark:sharpness-main-1}
Here the index $1$ is sharp in the sense that for $r<1$, $S_1-S_2$ is a compact operator $L^2(\R^n) \to H_{\oc}^{s,r}(\R^n)$ even if $V_1 \neq V_2$. This follows from \cite[Theorem~A.1]{jia2026metric-inverse} and Lemma~\ref{lemma:1c1c-compactness} below. 

On the other hand, for $V_1 \neq V_2$, $S_1-S_2$ is bounded from $H_{\oc}^{s,1}(\R^n)$ to $L^2(\R^n)$ but is not bounded from $H_{\oc}^{s,r}(\R^n)$ to $L^2(\R^n)$ for $r>1$.
In fact, by our proof of Theorem~\ref{thm:main-scmap-intro}, when $V_1 \neq V_2$, $A = S_1-S_2 \in I_{\oc-\oc}^{-1}(\R^n \times \R^n, (\mathrm{Gr}(\Cl))')$ has a principal symbol that does not vanish at a certain point. Then $A^*A \in \Psi_{\oc}^{-\infty,-2}$ (see Section~\ref{sec:the_1_cusp_pseudodifferential_algebra}) is elliptic at a certain point, which is unbounded $H_{\oc}^{s,r}(\R^n) \to H_{\oc}^{-s,-r}(\R^n)$ and this gives the unboundedness above.  
\end{remark}



\begin{thm} \label{thm:main-Poisson-intro}
For $n \geq 3$, if $\Poi_{\pm}^{(1)}-\Poi_{\pm}^{(2)}$ is compact from $L^2(\R^n)$ to $H^{\frac{3}{2},r}_{\ps}(\R^{n+1})$ for some $r \in R$, then $V_1=V_2$.
\end{thm}

\begin{remark} \label{remark:sharpness-main-2}
One can see from the proof of Proposition~\ref{prop:boundedness-compactness-1cps} that $\frac{3}{2}$ here is sharp in a similar fashion as Remark~\ref{remark:sharpness-main-1}: $m=3/2$ is the maximum such that $L^2(\R^n) \to H_{\ps}^{m,r}(\R^{n+1})$ is bounded when $V_1 \neq V_2$. Also, for $m<\frac{3}{2}$, $\Poi_{\pm}^{(1)}-\Poi_{\pm}^{(2)}$ is compact from $L^2(\R^n)$ to $H^{m,r}_{\ps}(\R^{n+1})$ even if $V_1 \neq V_2$.
This is because $\Poi_{\pm}^{(1)}-\Poi_{\pm}^{(2)} \in I_{\oc-\ps}^{-7/4}(\R^{n+1} \times \R^n, \Lagps_\pm)$ by \cite[Proposition~A.2]{jia2026metric-inverse}. 
So it is bounded $L^2(\R^n) \to H_{\ps}^{\frac{3}{2},r'}(\R^{n+1})$ for a fixed $r'>r$ by Proposition~\ref{prop:boundedness-compactness-1cps}.
The conclusion follows since the embedding $H_{\ps}^{\frac{3}{2},r'}(\R^{n+1}) \hookrightarrow H_{\ps}^{m,r}(\R^{n+1})$ is compact.
\end{remark}


\subsection{Strategy of the proof} \label{subsec:strategy-sketch}
We briefly summarize the strategy of the proof here.
The common theme in both inverse
results is that the compactness assumption is imposed at the critical order at which the first potential-dependent symbol appears. The compactness condition forces this symbol to vanish. 
For the Poisson operator this first non-trivial symbol contains the integral of $V_1-V_2$ along bicharacteristics up to a point. 
While for the scattering map, it contains an integral of the potential difference along the
corresponding bicharacteristic. The latter integral is then interpreted as a
weighted geodesic X-ray transform.

We first recall, in Sections~\ref{sec: parabolic sc PsiDO} and
\ref{sec:the_1_cusp_pseudodifferential_algebra}, the two pseudodifferential calculi which describe the two ends of the problem. The parabolic scattering calculus is adapted to the Schr\"odinger operator on spacetime, while the 1-cusp calculus is adapted to the asymptotic data appearing in
\eqref{eq:u expansion}. The relation between them is given by the source-sink
structure of the rescaled Hamilton flow of the principal symbol of $P$. After
blowing up the radial sets $\mathcal R_\pm$, the corresponding front faces are
canonically identified with $\ocphase$. Thus a one-cusp covector parametrizes a
bicharacteristic of the Schr\"odinger operator which enters and exits the
interaction region. At fiber infinity these bicharacteristics remain in a fixed
time slice and project to geodesics for the metric $g(t)$.

Section~\ref{sec:1c-ps-geometry-distribution} recalls the calculus of 1c-ps Fourier integral operators and Section~\ref{sec:1c-1c-calculus} recalls the calculus of 1c-1c Fourier integral operators.
The microlocalized Poisson operators are 1c-ps Fourier integral operators and the microlocalized scattering map is a 1c-1c Fourier integral operator.

We prove Theorem~\ref{thm:main-Poisson-intro} in Section~\ref{sec:Poisson-determine-potential}. 
Let $P_i=D_t+\Delta_{g(t)}+V_i$ and write $a_{\pm,j}^{(i)}$ for the symbols of the Poisson parametrices associated to $P_i$. 
Let $H_p^{2,0}$ be the rescaled Hamilton vector field associated to either of the $P_i$, which coincide at the leading order.
The principal symbols of them are obtained by solving transport equations along the $H_p^{2,0}$-flow. The leading symbols $a_{\pm,0}^{(i)}$ are independent of the potential. 

The first contribution of the potential difference occurs in
\begin{align}
    b_{\pm,1}:=a_{\pm,1}^{(1)}-a_{\pm,1}^{(2)} ,
\end{align}
which again satisfies a transport equation along the $H_p^{2,0}$ with $V_1-V_2$ appearing in the forcing term. We will then show $V_1-V_2=0$ by showing $b_{\pm,1}=0$.

The proof of Theorem~\ref{thm:main-scmap-intro} in
Section~\ref{sec:sc-map-determine-potential} is more involved. 
We will use the following identity relating the scattering map to Poisson operators:
\begin{align} \label{eq:S1-interms-Poisson}
    S_i Q_{\oc} = i(2\pi)^n (\Poip^{(i)})^* [P_i,Q_+]\Poim^{(i)}Q_{\oc},
\end{align}
where $Q_+,Q_{\oc}$ are suitable microlocal cutoffs. 
Using this, we will show that the sub-principal symbol of $(S_1-S_2)Q_{\oc}$ is given by an integral along the bicharacteristic line of $V_1-V_2$ times a certain weight.
We will then show that, with carefully chosen microlocal cut-offs, the weight is non-vanishing on the support of $V_1-V_2$.
Then the compactness assumption in Theorem~\ref{thm:main-scmap-intro} will enforce this sub-principal symbol to vanish, which in turn shows $V_1-V_2 = 0$. 

We conclude this subsection by discussing how an explicit reconstruction of the potential can be achieved using our method.
We begin with the Poisson operator, which is relatively more direct. 
Using \eqref{eq:b-1-formula}\eqref{eq:b+1-formula} with $V_1=V$, $V_2=0$, we know that from the principal symbol of $\Poipm^{(1)}-\Poipm^{(2)}$, one obtains a weighted integral of $V$ along a bicharacteristic line up to a certain point. Then differentiating this integral along the rescaled Hamilton vector field recovers $V$ times the weight. In addition, the weight is given by \eqref{eq:def-E}, which in principle can be computed explicitly in terms of the metric $g$ and this recovers $V$.

If one uses the data from the scattering map, then as sketched above,
applying our method to $V_2=0,V_1=V$,
the principal symbol of $S_1-S_2$ gives the weighted X-ray transform of the potential. 
Then as pointed out in \cite[Section~2]{uhlmann2016inverse}, one can construct the inverse of (the normal operator of) the X-ray transform using a Neumann series, which recovers the potential in our setting.

\subsection{Literature review}
Recovering a potential from scattering data is a classical problem in scattering theory and we can only give a brief review on a limited part of it. 
On the conceptual level, this is an inverse problem asking whether the interaction can be reconstructed from the observed incoming and outgoing data. 

The classical way to treat the scattering map is to  
decompose it with respect to energy. 
In one dimension, the Gel'fand--Levitan--Marchenko theory (see e.g. \cite{Gelfand-Levitan,Agranovich-Marchenko}) shows that the reflection/transmission data together with bound-state information determine the potential; see \cite{ChadanSabatier1989} and references therein. 
In higher dimensions the fixed-energy problem is more subtle. The works of Ramm, Nachman and Novikov \cite{Ramm-88,Nachman1992,Novikov1994} showed that, under compact-support or exponential decay hypotheses, a time-independent potential can be uniquely determined and in principle reconstructed from a fixed-energy scattering matrix or amplitude. 
See the lecture notes of Melrose \cite{Melrose-geometric-scattering} for the geometric microlocal analysis approach to this problem, which is closer to this article. In particular, in \cite[Section~3.6]{Melrose-geometric-scattering}, Melrose discussed how the leading order scattering amplitude associated to the Helmholtz operator recovers the X-ray transform of the potential, which in turn determines the potential.
For time-dependent potentials, see the works of Kitada and Yajima \cite{Kitada-Yajima-82}\cite{Kitada-Yajima-83}. See also the work of Enss-Weder \cite{EnssWeder1995} for the problem in the N-body setting, and Weder \cite{Weder-lecture-Nbody} in such a setting with a time-dependent potential.
For a similar problem for the wave equation, see the works of Stefanov \cite{Stefanov-wave} and Uhlmann \cite{Uhlmann-wave} and references therein.

The main advantage of our microlocal framework, compared with more classical scattering theory, is to handle curved spaces equipped with a time dependent metric and time-dependent potentials, on which the literature remains relatively sparse.
In the time-independent setting, this is treated by Joshi and Joshi-S\'{a} Berreto in a series of works \cite{Joshi-Sabarreto-shortrange,JS-magnetic,Joshi-Coulomb,Joshi-conformal} using microlocal tools developed by Melrose-Zworski \cite{melrose1996scattering}. 

For a similar inverse problem determining the metric instead of the potential, we refer readers to the discussion in \cite[Section~1.3]{jia2026metric-inverse}.
The calculus of Fourier integral operators we are using was introduced by Hassell and the author in \cite{HJ2026-scattering-map}.
It borrows ideas from H\"ormander \cite{FIO1}, Duistermaat-H\"ormander \cite{FIOII}, Melrose-Zworski \cite{melrose1996scattering}, Vasy \cite{vasy1998geometric}, Hassell-Wunsch \cite{hassell-wunsch2005schrodinger}. 

\section{Pseudodifferential algebras}

\subsection{The parabolic scattering pseudodifferential algebras}
\label{sec: parabolic sc PsiDO}
In this section we give a brief introduction of the parabolic scattering pseudodifferential algebra. 
As the name indicates, it is the `parabolic version' of the scattering pseudodifferential algebra, where `parabolic' refers to the way to compactify the fiber infinity. 
More precisely, let $(t,z,\tau,\zeta)$ be coordinates on $T^*\R^{n+1}$, with $\tau dt + \zeta \cdot dz$ being the canonical form.
Then the (compactified) parabolic scattering cotangent bundle, denoted by $\overline{^{\ps}T^*\R^{n+1}}$, is defined by compactifying $T^*\R^{n+1}$ in the following way.
The `base' $\R^{n+1}$ is compactified radially to be a ball with boundary defining function $x_{\ps} = (1+t^2+|z|^2)^{-1/2}$, while each fiber is compactified `parabolically' to be a ball with boundary defining function
\begin{align}\label{eq:rhops defn}
\rho_{\ps} = (1+\tau^2+|\zeta|^4)^{-1/4},
\end{align}
with $\ps$ standing for `parabolic scattering'. 

\begin{definition}
The symbol class $S_{\ps}^{m,l}(\R^{n+1})$, with $m,l$ being the differential and decay (in fact growth) orders respectively, is defined to be the space of $a \in C^\infty(T^*\R^{n+1})$ (in fact also their extensions to $\overline{^{\ps}T^*\R^{n+1}}$) such that
\begin{align*}
\|a\|_{ S_{\ps,N}^{m,l}} 
:= \sum_{|\alpha|+k+|\beta|+j \leq N} \sup_{T^*\R^{n+1}} |x_{\ps}^{l-|\alpha|-k} \rho_{\ps}^{m-\beta-2j} \partial_z^\alpha \partial_t^k \partial_\zeta^\beta \partial_\tau^j a(z,t,\zeta,\tau) | < \infty,
\end{align*}
for any $N \in \N$. And these norms give $S_{\ps}^{m,l}(\R^{n+1})$ a structure of a Fr\'echet space. In addition, when $\rho_{\ps}^m x_{\ps}^l a$ extends to a smooth function on $\overline{^{\ps}T^*\R^{n+1}}$, then we say $a$ is classical, and the corresponding symbol class is denoted by $S_{\ps,\cl}^{m,l}(\R^{n+1})$.
\end{definition}

Then the corresponding parabolic scattering pseudodifferential operators 
$\Psi_{\ps}^{m,l}(\R^{n+1})$ are operators that are quantizations of symbols in
$S_{\ps}^{m,l}(\R^{n+1})$, which means they have Schwartz kernels of the form
\begin{align*}
\Op(a) = (2\pi)^{-(n+1)} \int e^{ i (t-t')\tau+(z-z')\cdot \zeta}
a(t,z,\tau,\zeta) d\zeta d\tau,
\end{align*}
in the distributional sense.

The principal symbol of $A = \Op(a) \in \Psi_{\ps}^{m,l}(\R^{n+1})$ is defined to be the equivalence class of $a$ in $S_{\ps}^{m,l}(\R^{n+1})$ quotient by $S_{\ps}^{m-1,l-1}(\R^{n+1})$:
\begin{align*}
[a] \in S_{\ps}^{m,l}(\R^{n+1})/S_{\ps}^{m-1,l-1}(\R^{n+1}).
\end{align*}

A symbol $a \in S_{\ps}^{m,l}(\R^{n+1})$ (and corresponding operator $A=\Op(a)$) is said to be elliptic if it satisfies:
\begin{align*}
|\rho_{\ps}^m x_{\ps}^{l}a| \geq C, \text{ when } \rho_{\ps} \leq \epsilon 
\text{ or } x_{\ps} \leq \epsilon, 
\end{align*}
for some $\epsilon>0,C>0$. And we say that $a$ (and corresponding operator $A=\Op(a)$) is elliptic at $q \in \partial \overline{^{\ps}T^*\R^{n+1}}$ if there is a neighborhood of $q$ on which the above inequality is satisfied. And the set of all such $q \in \partial \overline{^{\ps}T^*\R^{n+1}}$ is denoted by $\Ell_{\ps}^{m,l}(a)$ (or $\Ell_{\ps}^{m,l}(A)$).

Finally we recall the concept of (parabolic) wavefront sets, which is also called the micro-support.
For $A = \Op(a)$, its parabolic scattering operator wavefront set $WF'_{\ps}(A)$ is defined, as a subset of $\partial \overline{^{\ps}T^*\R^{n+1}}$, as follows.
For $q \in \partial \overline{^{\ps}T^*\R^{n+1}}$, we say
$q \notin \WF'_{\ps}(A)$ if and only if there is $\chi \in C^\infty(\overline{^{\ps}T^*\R^{n+1}})$ with $\chi(q) = 1$ such that $\chi a \in \mathcal{S}(\overline{^{\ps}T^*\R^{n+1}})$.
In particular, $\WF'_{\ps} (A) = \emptyset$ if and only if $A \in \Psi_{\ps}^{-\infty,-\infty}(\R^{n+1})$.

\subsection{The 1-cusp pseudodifferential algebra}
\label{sec:the_1_cusp_pseudodifferential_algebra}


In this section, we briefly introduce the 1-cusp pseudodifferential algebra, and refer readers to \cite{Zachos:Thesis}\cite[Section~2]{zachos2022inverting}\cite[Section~2]{jia2022tensorial} for more details.
For its connection to time-dependent Schr\"odinger equations, see \cite[Section~3.3]{HJ2026-scattering-map}.

Let $M$ be an $m-$dimensional manifold with boundary, with boundary defining function $x_{\oc}$, and suppose $y=(y_1,..,y_{m-1})$ is a coordinate system of $\partial M$, which together with $x_{\oc}$ forms a local coordinate system of $M$ near the boundary. Then the space of 1-cusp vector fields, denoted by $\mathcal{V}_{\oc}$, is locally spanned over $C^\infty(M)$ by
\begin{align} \label{eq: 1-cusp vector fields, local}
x_{\oc}^3\partial_{x_{\oc}}, x_{\oc}\partial_{y_j}, j=1,2,...,m-1.
\end{align}
$\mathcal{V}_{\oc}$ gives rise to a vector bundle with \eqref{eq: 1-cusp vector fields, local} being its local frame, which is called 1-cusp tangent bundle, and denoted by $^{\oc}TM$.

The class of 1-cusp differential operators of order (at most) $k$, denoted by $\mathrm{Diff}_{\oc}^k(M)$, consists of polynomials of these vector fields of degree at most $k$:
\begin{align}  \label{eq: Diff 1c definition}
\begin{split}
\mathrm{Diff}_{\oc}^k(M) 
= \{ \sum_{\alpha+|\beta| \leq k}  a_{\alpha\beta}(x_{\oc},y)(x_{\oc}^3\partial_{x_{\oc}})^\alpha(x_{\oc}\partial_{y})^\beta : 
\\ a_{\alpha \beta} \in C^\infty(M), \alpha \in \N, \beta \in \N^{m-1} \}.
\end{split}
\end{align}

The 1-cusp cotangent bundle, denoted by $^{\oc}T^*M$ is the dual bundle of $^{\oc}TM$. It is locally spanned over $C^\infty(M)$ by 
\begin{align}
\frac{dx_{\oc}}{x_{\oc}^3}, \frac{dy_j}{x_{\oc}}, j=1,2,...,m-1.
\end{align}
Then $T^*M$ embeds into $^{\oc}T^*M$ canonically in the interior of $M$, giving it a symplectic structure naturally. In particular, one may write the canonical one form as
\begin{align}  \label{eq: 1c- canonical form}
\xi_{\oc} \frac{dx_{\oc}}{x_{\oc}^3} + \eta_{\oc} \cdot \frac{dy}{x_{\oc}},
\end{align}
where $(\xi_{\oc},\eta_{\oc})$ are coordinates of fibers on $^{\oc}T^*M$. 
We use $\overline{^{\oc}T^*M}$ to denote the compactification of $^{\oc}T^*M$, obtained by compactifying each fiber radially to be a ball
with boundary defining function
\begin{align}
\rho_{\oc} = (1+\xi_{\oc}^2+|\eta_{\oc}|^2)^{-1/2}.
\end{align}

Notice the minor difference with $\overline{^{\ps}T^*\R^{n+1}}$, which is compactified on both base and fiber level, is because the 1-cusp construction is happening on the manifold with boundary $M$, which is already `compactified a priori'.

The 1-cusp symbol class of differential order $m$ and decay (in fact growth) order $l$, denoted by $S_{\oc}^{m,l}(M)$, is defined to be smooth functions on ${}^{\oc}T^*M$ satisfying that for any $j,\alpha,k,\beta$ there exists a constant $C_{j\alpha k \beta}$ such that
\begin{align}
 |x_{\oc}^l \rho_{\oc}^{m-k-|\beta|} (x_{\oc}\partial_{x_{\oc}})^j\partial_y^\alpha\partial_{\xi_{\oc}}^k\partial_{\eta_{\oc}}^\beta a(x_{\oc},y,\xi_{\oc},\eta_{\oc})|
 \leq C_{j\alpha k \beta}.
\end{align}
And locally they quantize to be operators acting by
\begin{align}
\Op(a)u(x_{\oc},y) = (2\pi)^{-n}
\int e^{ i \xi_{\oc}\frac{x_{\oc}-x_{\oc}'}{x_{\oc}^3}+\eta_{\oc} \cdot \frac{y-y'}{x_{\oc}}}
a(x_{\oc},y,\xi_{\oc},\eta_{\oc}) u(x_{\oc}',y') d\xi_{\oc}d\eta_{\oc} \frac{dx_{\oc}'dy'}{(x_{\oc}')^{n+2}}.
\end{align}
The collection of all such operators is called 1-cusp pseudodifferential operators with differential order $m$ and decay order $l$, and denoted by $\Psi_{\oc}^{m,l}(M)$.
The operator wavefront set $\WF_{\oc}'(A)$ is the part of $\partial \overline{{}^{\oc}T^*M}$ defined by 
\begin{align}
\begin{split}
\WF'_{\oc}(A) = \Big(\{ q \in \partial \big( \overline{{}^{\oc}T^*M})\; | \text{There exists } \chi \in C^\infty(\overline{{}^{\oc}T^*M}) \text{ such that } \chi a \in \mathcal{S}(\overline{{}^{\oc}T^*M}) , \, \chi(q)=1 \}  \Big)^\complement. 
\end{split}
\end{align}

Next we define the ellipticity of symbols and operators.
\begin{definition} \label{eq:def-1c-ellipticity}
A symbol $a \in S^{m,l}_{\oc}(M)$ is called elliptic if
\begin{align*}
|a(x_{\oc},y,\xi_{\oc},\eta_{\oc})| \geq cx_{\oc}^{-l} \la (\xi_{\oc},\eta_{\mathrm{1c}}) \ra^m, \quad c>0 \text{ when }  |(\xi_{\oc},\eta_{\oc})|^{-1} \leq \epsilon \text{ or } x_{\oc} \leq \epsilon,
\end{align*}
for some $\epsilon>0,c>0$, and its quantization $A$ is also called elliptic in this case.
\end{definition}

Under the elliptic condition in Definition~\ref{eq:def-1c-ellipticity}, see \cite[Section~2.5]{zachos2022inverting}, its quantization $A$ has a parametrix $B \in \Psi_{\oc}^{-m,-l}(M)$ such that
\begin{align*}
AB-\Id, \; BA-\Id \in \Psi_{\oc}^{-\infty,-\infty}(M).
\end{align*}

One can now define 1-cusp Sobolev spaces
$H^{s,r}_{\oc}(M)$ (see \cite[Section~2.5]{zachos2022inverting}) for $s \geq 0$ by choosing $A \in \Psi_{\oc}^{s,0}(M)$ elliptic, and demanding
\begin{equation} \label{eq:1c-Sobolev-def}
u\in H^{s,r}_{\oc}(M) \Leftrightarrow u\in x_{\oc}^r L_{\oc}^2(M)\ \text{and}\ Au\in x_{\oc}^r L_{\oc}^2(M);
\end{equation}
here $L_{\oc}^2(M)$ is the $L^2$ space relative to the 1-cusp density $\frac{dx\,dy}{x_{\oc}^{n+2}}$.
They are equipped with norms:
$$
\|u\|^2_{H^{s,r}_{\mathrm{1c},h}}=\|x_{\oc}^{-r}u\|_{L_{\oc}^2}^2+\sum_{j+|\alpha|\leq s}\|(hx_{\oc}^3D_{x_{\oc}})^j(h^{1/2}x_{\oc}D_y)^\alpha u\|_{L_{\oc}^2}^2.
$$
The 1-cusp Sobolev spaces for other $s$ are defined via interpolation and duality.
Then 1-cusp pseudodifferential operators are bounded on these Sobolev spaces, namely for $A \in \Psi_{\oc}^{m,l}(M)$ and all $s,r$, $A$ is a bounded linear operator from $H^{s,r}_{\oc}$ to $H^{s-m,r-l}_{\oc}$.

\section{\texorpdfstring{The $\oc-\ps$ geometry and analysis}{The 1c-ps geometry and analysis}} 
\label{sec:1c-ps-geometry-distribution}

In this section, we recall the main geometric and analytic ingredients needed for the theory of $\oc-\ps$ Fourier integral operators from \cite[Section~3]{gell2022propagation} and \cite[Section~4, Section~5]{HJ2026-scattering-map}. This is used to characterize the Poisson operator sending the asymptotic data $f_\pm$ in \eqref{eq:u expansion} to the solution $u$.

\subsection{The link between the 1c-phase space and the Schr\"odinger equation}
\label{subsec:bulk-boundary}

We recall the phase space dynamics associated with the time-dependent Schr\"odinger equation and discuss how $\ocphase$ parametrizes its bicharacteristic lines.

Let $p$ be the left symbol of $P$, the interesting part is where the Schr\"odinger operator is non-elliptic, i.e., the characteristic variety:
\begin{equation}
    \Sigma(P):= p^{-1}(0) \subset \overline{{}^{\ps}T^*\R^{n+1}}.  
\end{equation}
The set 
\begin{equation} \label{eq:CharP-def}
    \Char(P) = \Sigma(P) \cap \partial(\overline{{}^{\ps}T^*\R^{n+1}})
\end{equation}
is where the microlocal propagation takes place.

One of the key ingredients of our analysis is the global source-sink structure of the rescaled Hamilton flow associated to $P$, which is denoted by $H_p^{2,0}$ in \eqref{eq:Hp20-def} below.
Its integral curves (more precisely, their closures) are called bicharacteristic lines.
As the discussion in \cite[Section~3]{gell2022propagation} shows, 
the bicharacteristic lines starting on the boundary with finite frequency will remain on the boundary for all time and only those bicharacteristics at fiber infinity of $\overline{{}^{\ps}T^*\R^{n+1}}$ will leave $\partial \overline{\R^{n+1}}$ and meet the metric perturbation, so we will restrict ourselves to be near the characteristic set in such a region.
In addition, such a bicharacteristic line will stay in the same time slice, or in other words, it reaches the endpoint `instantly' in terms of $t$.
In this region, we know $|\tau|^{1/2}$ and $|\zeta|_g$ are comparable and we can use $\rhops = \big( \sum_{j,k} g^{jk}(z, t) \zeta_j \zeta_k \big)^{-1/2}$.


Let $\hat{\tau} = \rho_{\ps}^2\tau \in \R$, and $\hat{\zeta} \in \R^{n-1}$ be coordinates parametrizing $\rho_{\ps} \zeta$ in the sphere with respect to $g$. A valid coordinate system with $(t,z)$ in a compact region is
\begin{equation} \label{eq:ps-coordinates}
    (t,z,\rho_{\ps},\hat{\tau},\hat{\zeta}).
\end{equation}

Let $\rho_{\base} = (1+t^2+|z|^2)^{-1/2}$ be the defining function of the spacetime infinity, then the rescaled Hamilton vector field of $p$
\begin{equation} \label{eq:Hp20-def}
    H_p^{2,0} = \rhops \rho_{\base}^{-1}H_p
\end{equation}
is a smooth vector field on $\overline{{}^{\ps}T^*\R^{n+1}}$ that is tangent to its boundary. 
The flow of $H_p^{2,0}$ has a global source-sink structure with the following `radial sets' being the source and sink.
\begin{definition}
 \label{defn: radial sets}
The radial set (of $P$) $\mathcal{R}$ is defined to be 
\begin{align}
\mathcal{R} = \{ q \in \Char(P): H^{2,0}_p \text{ vanishes at } q \}.
\end{align}
\end{definition}
Since the flow of $H^{2,0}_p$ has a source-sink structure, we have the decomposition of $\mathcal{R}$:
\begin{equation}
    \mathcal{R} = \Radm \cup \Radp \subset \overline{{}^{\ps}T_{\partial \overline{\R^{n+1}}}^*\R^{n+1}},
\end{equation}
where $\Radm$ is the source and $\Radp$ is the sink. 
In the region $|t/ z | \leq C$ for a constant $C>0$, they are given by 
\begin{equation}\label{eq:rad.corner.pm}
\mathcal{R}_\pm \supset \{ x_{\ps} = 0, \ t/|z| = \pm \rho_{\ps}/2, \ \hat \zeta/|\zeta| = \pm z/|z|, \ \tau/|\zeta|^2 = -1 \} \cap \{  \frac{|t|}{|z|} \leq C \}.
\end{equation}

So one can see that $\Radm$ is a graph over the southern hemisphere (i.e., the part of $\partial \overline{\R^{n+1}}$ with $t/\la z \ra \leq 0$) and $\Radp$ is a graph over the northern hemisphere (i.e., the part of $\partial \overline{\R^{n+1}}$ with $t/\la z \ra \geq 0$).
Both of them turn vertical (in the sense of tending to fiber infinity) when approaching the equator $\{ t/|z| = 0, 1/|z|=0 \}$.
Though they overlap at the equator on the base level,  they tend to the opposite direction in terms of the frequency, hence they remain disjoint on the phase space level.

As aforementioned, for each point on $\partial \overline{\R^{n+1}}$, there is an $n$-dimensional family of geodesics tending to it. On a phase space level, this corresponds to different bicharacteristic lines tending to or emanating from the same point in $\mathcal{R}_\pm$. Then they are distinguished via blowing up $\mathcal{R}_\pm$ within $\Sigma$:
\begin{equation}
    [\Sigma; \mathcal{R}_\pm].
\end{equation}
We denote the front face created by the blow up by $W_\pm$, then $W_\pm$ precisely parametrizes all bicharacteristic lines. Using $W_\pm$, we have the following correspondence between a point on a bicharacteristic line and the `endpoint' of this bicharacteristic line. 
\begin{prop}\cite[Lemma~3.4]{HJ2026-scattering-map}
\label{prop:Wpm-1c-identification}
$W_\pm$ is canonically diffeomorphic to $\overline{{}^{\oc}T^*\mathcal{R}_\pm}$, the radially compactified 1-cusp cotangent bundle over $\mathcal{R}_\pm$.
\end{prop}
Since we have a canonical identification between $\mathcal{R}_\pm$ and $\overline{\R^n}$ via projection to the base, $\overline{{}^{\oc}T^*\mathcal{R}_\pm}$ above can be replaced by $\overline{{}^{\oc}T^*\R^n}$. See \cite[Section~3.3]{HJ2026final} for a more explicit characterization of this correspondence.

\subsection{Geometry of the 1c-ps phase space}
The phase space for 1c-ps Lagrangian distributions (a class of operators that includes suitably microlocalized Poisson operators) is obtained by blowing up the corner of
\begin{align} \label{eq:calM0-def-product-bundle}
\mathcal{M}_0 =  \overline{^\ps{T^*\RR^{n+1}}} \times \overline{^\oc{T^*\RR^n}} 
\end{align}
at base infinity of $\overline{^\oc{T^*\RR^n}}$ and fiber-infinity of $\overline{^\ps{T^*\R^{n+1}}}$. Here the $\RR^n$ factor represents the interior of either $\Rp$ or $\Rm$.  We refer readers to \cite{melrose1993atiyah}\cite{Melrose1994} for more details about blow ups. Concretely, we define
\begin{align} \label{eq: 1c-ps cotangent bundle definition}
\mathcal{M}:= [ \mathcal{M}_0 ; \{ \rho_{\ps} = 0, x_{\oc} =0 \} ], \quad \rhops = \big( \sum_{j,k} g^{jk}(z, t) \zeta_j \zeta_k \big)^{-1/2}, 
\end{align}
and denote the blow down map by 
\begin{align}
\beta_{\oc-\ps}: \mathcal{M} \rightarrow \mathcal{M}_0.
\end{align}
The front face created by this blow-up shall be denoted $\ffocps$. 
The new smooth coordinate on $\ffocps$ introduced by the blow up is 
\begin{equation}\label{eq:sigma}
\varsigma = \frac{x_{\oc}}{\rho_{\ps}},
\end{equation}
or its reciprocal. In the interior of $\ffocps$, either $\xoc$ or $\rhops$ can be taken as a boundary defining function. 

The manifolds $\SM_0$ and $\SM$ have codimension 4 corners. However, we shall only be interested in a neighbourhood of a compact subset $K$ of the interior of $\ffocps$; in particular, we shall stay away from all the other boundary hypersurfaces. So, in effect, we are dealing with a manifold with boundary. 

The manifold $\mathcal{M}$ is endowed with a canonical symplectic structure\footnote{In this article, we allow symplectic structures to blow up or degenerate at the boundary} $\omega$ from $\mathcal{M}_0$ by lifting the symplectic form on $\mathcal{M}_0$, which in turn is equipped with the product symplectic structure from its two factors. We are particularly interested in the symplectic/contact structures on $\ffocps$ induced by this symplectic structure in the interior. To prepare for this, we shall specify coordinates to use in a neighbourhood of $K \subset \ffocps$. These will be $\xoc$ (a boundary defining function) and $\yoc$, which are base coordinates on $\RR^n$ near base infinity;  $\xioc$, $\etaoc$, their one-cusp dual coordinates as defined in Section~\ref{sec:the_1_cusp_pseudodifferential_algebra}; $z$, $t$, Euclidean space and time coordinates;  $\varsigma$ as defined in \eqref{eq:sigma}; and fibre coordinates near fibre-infinity, which we take to be $\tilde \tau = \tau \xoc^2$, $\hat \zeta = \zeta / |\zeta|$. We remark that we are also mostly interested in $(z, t)$ near the perturbation of the metric, which by assumption is a compact set in spacetime. On the other hand, $(\zeta, \tau)$ will be near infinity, since $\rhops = 0$ at $\ffocps$. To summarize, our coordinates are
\begin{equation}\label{eq:coordinates near ffocps}
\xoc, \quad \yoc, \quad \xioc, \quad \etaoc, \quad z, \quad t, \quad \varsigma, \quad \tilde \tau, \quad \hat \zeta.
\end{equation}


\begin{definition}[Admissible 1c-ps Lagrangian submanifold and 1c-ps fibred-Legendre submanifold]  \label{defn: admissible 1c-ps Lagrangian submanifold and 1c-ps fibred-Legendre submanifold}
We define an admissible 1c-ps Lagrangian submanifold of $\SM$ to be a $2n+1$-dimensional submanifold $\Lambda$ that is Lagrangian in the interior (the canonical symplectic form $\omega$ vanishes on it), such that 
\begin{itemize}
\item $\Lambda$ meets $\ffocps$ transversally, 
\item the differential $dt$ is non-vanishing on $\Lambda \cap \ffocps$, and 
\item  its closure is disjoint from all other boundary hypersurfaces of $\SM$ (other than $\ffocps$). 
\end{itemize}
We define a fibred-Legendre submanifold $L$ of $\ffocps$ to be the boundary of an admissible 1c-ps Lagrangian submanifold. In other words, there exists $\Lambda$ as above such that $L = \Lambda \cap \ffocps$. 
\end{definition}
See \cite[Proposition~4.4]{HJ2026-scattering-map} for the reason for the term `fibred' from a symplectic fibration.
The 1c-ps Lagrangian submanifolds that are used for our analysis are the twisted forward and backward sojourn relations which arise from the bulk-boundary duality we discussed in Section~\ref{subsec:bulk-boundary}.
As the calculation in \cite[Section~4.4]{HJ2026-scattering-map} shows, the rescaled Hamilton vector field $H_p^{2,0}$ is a smooth vector field tangent to the boundary of $[\Sigma; \SR]$ except at $W_\pm$, where it is transverse: inward-pointing at $W_-$ and outward-pointing at $W_+$. Using the non-trapping assumption we see that each point of $W_\pm$ gives rise to a smooth integral curve of  $H_p^{2,0}$, that travels from $W_-$ to $W_+$ in finite parameter time $s$. Let $q_-$ be a point of $W_-$, and let  $\gamma_{q_-}(s)$ be the integral curve of $H_p^{2,0}$ emanating from $q_-$ at time $s=0$, and arriving at $W_+$ at time $T(q_-) > 0$. Similarly, let $q_+$ be a point of $W_+$, and let $\mu_{q_+}(s)$ be the integral curve of $H_p^{2,0}$ emanating from $q_+$ at time $s=0$, and arriving at time $T'(q_+) < 0$. 
The forward sojourn relation is the subset of $\SM_0 =  \overline{^\ps{T^*\RR^{n+1}}} \times \overline{^\oc T^*{\Rm}} $ defined by 
\begin{equation}\label{eq:FSR def}
\FSR = \{ (\gamma_{q_-}(s), q_-) \mid q_- \in W_- , \ s \in [0, T(q_-)] \}. 
\end{equation}
Similarly, the backward sojourn relation is defined by 
\begin{equation}\label{eq:BSR def}
\BSR = \{ (\mu_{q_+}(s), q_+) \mid q_+ \in W_+ , \ s \in [T'(q_+), 0] \}. 
\end{equation}

Then our Lagrangian will be a lifted and microlocalized version of them, after twisting the sign of the 1-cusp frequencies.
Let $U_\pm \subset W_\pm \sim \overline{ ^{\oc}T^* \R^n} $ be open sets disjoint from fibre-infinity, and let $G_\pm \subset \overline{ ^{\ps}T^* \R^{n+1}}$ be open sets disjoint from spacetime infinity. Consider the microlocalized Lagrangians 
\begin{equation}\label{eq:microlocalized sojourn relns}
\Lambda_\pm = \beta_{\oc-\ps}^*\big(\{ (\gamma_{q_\pm}(s), -q_{\pm}) \in \overline{\Lambda_\pm} \mid q_{\pm} \in U_\pm, \ \gamma_{q_\pm}(s) \in G_\pm  \}\big)
\end{equation}
where $-q_\pm$ means changing the sign of the fiber part and we now view it as being in $\SM$ rather than $\SM_0$. That is, we take these sets to be the closure, in $\SM$, of their interiors lifted to $\SM$ from $\SM_0$ via the blowdown map $\beta_{\oc-\ps}$. 
This $\Lambda_\pm$ depends on the choice of $U_\pm$ and $G_\pm$ but we do not indicate this in the notation, regarding these choices of open sets as fixed. 
As shown in \cite[Proposition~4.12]{HJ2026-scattering-map}, this $\Lambda_\pm$ is admissible in the sense of Definition~\ref{defn: admissible 1c-ps Lagrangian submanifold and 1c-ps fibred-Legendre submanifold}.

The discussion above also gives rise to the following map $\msf{I}_\pm$ that we will use later.
For $\tilde{q} \in \Char(P)$, there is a unique $q_\pm \in W_\pm$ such that
\begin{equation}
(\tilde{q}, q_\pm) \in \overline{\Lambda'_\pm}.
\end{equation}
That is, $q_\pm$ is the endpoint of the $H_p^{2,0}$-flow starting from $\tilde{q}$ in the forward/backward direction.
Then we view $q_\pm$ as a point in $\overline{{}^{\oc}T^*\R^n}$ and define
\begin{equation} \label{eq:msf-I-def}
    \msf{I}_\pm(\tilde{q}) = q_\pm \in \overline{{}^{\oc}T^*\R^n}.
\end{equation}
In addition, for $\tilde{q}$ at fiber infinity of $\psphase$, we have $\msf{I}_\pm(\tilde{q}) \in \overline{{}^{\oc}T_{\partial\overline{\R^n} }^*\R^n}$.

We refer readers to \cite[Definition~4.5]{HJ2026-scattering-map} for the detailed definition of the parametrization of 1c-ps fibred Legendre submanifolds and only give a typical form of it here.
In addition, using \cite[Lemma~B.1]{HJ2026-scattering-map}, after a linear change of coordinates in $z$ (with $\tilde \zeta$ transforming correspondingly) we have a splitting of coordinates so that in a neighborhood of $q_0 \in \Legps$ the projection
 \begin{align}\label{eq:fullrank}
    \Lambda_\pm \ni (t,z,\tilde{{\tau}},\tilde{\zeta},x_{\oc},y_{\oc},\xi_{\oc},\eta_{\oc}) \rightarrow      (t,z',\tilde{{\zeta}}'',x_{\oc},y_{\oc})
 \end{align}
has full rank $2n+1$, where $z'=(z_1,...,z_k),\ \tilde{\zeta}''=(\tilde{\zeta}_{k+1},...,\tilde{\zeta}_n)$ for some $k \in \{ 1, \dots, n\}$. Geometrically, $\tilde{\zeta}''$ is present due to the degeneracy of the exponential map, or equivalently the presence of conjugate points.

Under this splitting of coordinates, using \cite[Proposition~4.6, Lemma~B.2]{HJ2026-scattering-map} the phase function used to parametrize $\Legps$ (together with its Lagrangian extension $\Lambda_\pm$) can be taken as the following normal form:
\begin{align} \label{eq:1c-ps-phase-normal-form}
\Phi_{\oc-\ps} = -\frac{t}{x_{\oc}^2}+\frac{z'' \cdot \tilde{\zeta}''-\tilde{\varphi}_1(t,z',\tilde{{\zeta}}'',x_{\oc},y_{\oc}) }{x_{\oc}}.
\end{align}

\subsection{The calculus of 1c-ps Fourier Integral Operators}
\label{subsec:1c-ps-calculus}
In this subsection, we recall the calculus of 1c-ps Fourier Integral Operators developed in \cite[Section~5]{HJ2026-scattering-map}.

\begin{definition} \label{defn: 1c-ps FIO}
Let $\Legps$ be a fibered-Legendre submanifold of $\SM$ and $\Lagps$ be its Lagrangian extension in Definition~\ref{defn: admissible 1c-ps Lagrangian submanifold and 1c-ps fibred-Legendre submanifold}. 
A 1c-ps Legendre distribution associated to  $\Legps$
of order $m$ is a distributional half-density that can be written (modulo $\mathcal{S}(\R^{n+1} \times \R^n)$) as a finite sum of oscillatory integrals of the form 
\begin{multline} \label{eq: 1c-ps FIO definition}
u(\msf{K}) =  (2\pi)^{- ( \frac{2n+1}{4} ) - \frac{k_0+k_1}{2} }
\Big(\int e^{ i( \frac{ \varphi_0(t,\theta_0) }{x_{\oc}^2} + \frac{\varphi_1(\msf{K}',\theta_0,\theta_1)}{x_{\oc}} )} 
x_{\oc}^{-(m+\frac{2k_0+k_1}{2})-\frac{1}{4}}
\\  \times a(\msf{K},\theta_0,\theta_1) d\theta_0d\theta_1 \Big)
|dtdz|^{1/2}|\frac{dx_{\oc}d\yoc}{x_{\oc}^{n+2}}|^{1/2}, \quad \msf{K} = (\xoc, \yoc, t, z), \quad \msf{K}' = (\yoc, t, z),
\end{multline}
where $\Phi_{\oc-\ps} = \frac{ \varphi_0(t,\theta_0) }{x_{\oc}^2} + \frac{\varphi_1(\msf{K},\theta_0,\theta_1)}{x_{\oc}}$ is a parametrization of $\Legps$, 
with $\theta_0 \in \R^{k_0}$, $\theta_1 \in \R^{k_1}$, and 
 $a \in C_c^\infty([0,\epsilon)_{\xoc} \times \R^{n-1}_{\yoc} \times \R^{n+1}_{t,z} \times \R^{k_0 + k_1}_{\theta_0, \theta_1})$ is assumed to be supported in the region where $\Phi_{\oc-\ps}$ parametrizes $\Legps$. 
 The set of such Legendre distributions is denoted $I_{\oc-\ps}^m(\R^{n+1} \times \R^n, \Lagps)$. 
 A linear operator $A$, mapping half-densities on $\R^n$ to half-densities on $\R^{n+1}$ is called a 1c-ps Fourier Integral Operator of order $m$ associated to $\Legps$ if its Schwartz kernel is a Legendre distribution of order $m$ associated to $\Legps$. 
\end{definition}

As shown in \cite[Proposition~5.4]{HJ2026-scattering-map}, the oscillatory integral in the definition above can be written as, modulo a Schwartz error, an oscillatory integral using another parametrization, on the region where both parametrizations are valid.

Let $S^{[m]}_{\oc-\ps}(\Legps)$ be the bundle incorporating the bundle $|N^*(\ffocps)|^{-m-\frac{2n+5}{4}}$, which models a section that is homogeneous of degree $-m-\frac{2n+5}{4}$ in $x_{\oc}$ and the choice of the boundary defining function and the Maslov bundle and $\Omega(\Legps)$ be the half-density bundle on $\Legps$. See \cite[Equation~(5.6)]{HJ2026-scattering-map} for more details.
Then the principal symbol map is 
\begin{align} \label{eq: invariant, symbol map, 1c-ps}
\begin{split}
\sigma^m_{\oc-\ps}: \quad   I^m_{\oc-\ps}(\R^{n+1} \times \R^n ,\Lagps; \Omega^{1/2}) 
\rightarrow 
 C^\infty(\Legps \cap \partial \mathcal{M}; 
\Omega^{1/2}(\Legps) \otimes S_{\oc-\ps}^{[m]}(\Legps)).
\end{split}
\end{align}
This principal symbol map gives the following short exact sequence in \cite[Proposition~5.5]{HJ2026-scattering-map}, which is the key ingredient to establish the calculus of 1c-ps Fourier integral operators.

\begin{align} \label{eq:1c-ps-exact-sequence}
\begin{split}
0 & \rightarrow I^{m-1}_{\oc-\ps}(\R^{n+1} \times \R^n,\Lagps)
\rightarrow I^{m}_{\oc-\ps}(\R^{n+1} \times \R^n, \Lagps)
 \xrightarrow{\sigma^m_{\oc-\ps}} C^\infty(\Legps \cap \partial \mathcal{M}; 
\Omega^{1/2}(L) \otimes S_{\oc-\ps}^{[m]}(\Legps)) \rightarrow 0.
\end{split}
\end{align}

Next we recall how pseudodifferential operators act on 1c-ps Fourier integral operators.
First we consider the case when we compose a ps-pseudodifferential operator from the left.
\begin{prop}{\cite[Theorem~5.7]{HJ2026-scattering-map}} 
\label{prop: PsiDO- 1c-ps FIO composition}
Suppose $Q \in \Psi_{\ps}^{m',0}(\R^{n+1})$ has parabolically homogeneous principal symbol at fibre-infinity. If $A \in I_{\oc-\ps}^{m}(\R^{n+1} \times \R^n, \Lagps; \Omega^{1/2})$, then we have  
\begin{equation}
QA \in I^{m+m'}_{\oc-\ps}(\R^{n+1} \times \R^n,\Lagps),
\end{equation}
with principal symbol 
\begin{align}  \label{eq: QA principal symbol, non-vanishing q, density version}
\sigma^{m+m'}_{\oc-\ps}(QA) = \sigma^{m'}_{\ps}(Q)|_{\Legps} \otimes \sigma^m_{\oc-\ps}(A),
\end{align}
where we lift the principal symbol of $Q$ to $\SM$, and view it as a section of $|N^*(\ffocps)|^{-m'}$ over $\Legps$. 
\end{prop}

When the pseudodifferential operator has vanishing principal symbol on $\Lagps$, then we have the following refined characterization.
\begin{prop}{\cite[Theorem~5.8]{HJ2026-scattering-map}} \label{prop: vanishing principal symbol product}
Suppose $P \in \Psi_{\ps}^{m',0}(\R^{n+1})$, and its parabolically homogeneous principal symbol $p_{\hom}$ vanishes identically on the projection of $\Legps$ in $\psphase$. In addition, let $p$ be its left full symbol and assume that $\tilde{p}(t,z,\tilde{\tau},\tilde{\zeta})=x_{\oc}^{m'}p(t,z,\tau,\zeta)$ is smooth.\footnote{This in particular is satisfied by all differential operators.} 
If $A \in I_{\oc-\ps}^{m}(\R^{n+1} \times \R^n,\Lambda_\pm; \Omega^{1/2})$, then we have 
\begin{align*}
PA \in I^{m+m'-1}_{\oc-\ps}(\R^{n+1} \times \R^n,\Lambda_\pm; \Omega^{1/2}).
\end{align*}
The principal symbol of $PA$ is as follows: let 
\begin{align*}
 \sigma^m_{\oc-\ps}(A) = \textbf{a} \otimes |dx_{\oc}|^{-m-\frac{2n+5}{4}},
\end{align*}
with $\textbf{a}$ being a section of $\Omega^{1/2}(\Legps) \otimes M(\Legps)$, then
\begin{align} \label{eq: 1c-ps vanishing principal composition, principal symbol} 
\begin{split}
\sigma_{\oc-\ps}^{m+m'-1}(PA) = & (-i\mathscr{L}_{ \rHp} + i (\frac{m'-1}{2} + m + \frac{2n+5}{4} )(x_{\oc}^{-1} \rHp x_{\oc}) + p_{\sub})\textbf{a}
\\& \otimes |dx_{\oc}|^{-m-m'+1-\frac{2n+5}{4}}.
\end{split}
\end{align}
\end{prop}

We can also compose 1c-pseudodifferential operators from the right. As one would expect from the duality in Section~\ref{subsec:bulk-boundary}, the roles of the differential and decay orders are switched compared with the $\ps$-setting.
\begin{prop}{\cite[Theorem~5.9]{HJ2026-scattering-map}} \label{prop: PsiDO- 1c-ps FIO composition-1c-right}
Suppose $Q' \in \Psi_{\oc}^{-\infty,m'}(\R^n)$. If $A \in I_{\oc-\ps}^{m}(\R^{n+1} \times \R^n, \Lagps)$ 
then we have  
\begin{equation}
AQ' \in I^{m+m'}_{\oc-\ps}(\R^{n+1} \times \R^n,\Lagps),
\end{equation}
with principal symbol 
\begin{align}  \label{eq: AQ' principal symbol, non-vanishing q, density version}
\sigma^{m+m'}_{\oc-\ps}(AQ') =   \sigma^m_{\oc-\ps}(A)
\otimes \sigma^{m'}_{\oc}(Q')|_{\Legps}
\end{align}
where we lift the principal symbol of $Q'$ to $\SM$, and view it as a section of $|N^*(\ffocps)|^{-m'}$ over $\Legps$. 
\end{prop}

Next we give a characterization of Poisson operators as 1c-ps Fourier integral operators.
The (forward and backward) Poisson operators $\mathcal{P}_\pm$ are the operators sending the `final state data' $f_\pm$ in \eqref{eq:u expansion} to the solution $u$.
With this calculus of 1c-ps Fourier integral operators, we have the following characterization of these two Poisson operators. 
We begin by introducing pseudodifferential operators microlocalizing the Poisson operator to the part associated to $\Lambda_\pm$, which is the part meeting the geometric perturbation.
Let $U_\pm$ and $G_\pm$ be as in \eqref{eq:microlocalized sojourn relns}, we choose 
\begin{equation} \label{eq:Q1c-Qps-def}
    Q_{\oc} \in \Psi_{\oc}^{0,0}(\R^n), \quad Q_{\ps} \in \Psi_{\ps}^{0,0}(\R^{n+1})
\end{equation}
such that $\WF_{\oc}'(Q_{\oc}) \subset U_-$ (hence disjoint from fiber-infinity in $\overline{{}^{\oc} T^* \R^n}$), and $\WF_{\ps}'(Q_{\ps}) \subset G_-$ (hence is away from spacetime infinity). We further assume that 
$Q_{\oc}$ is microlocally equal to the identity on all points in $\overline{{}^{\oc} T^* \R^n}$ whose corresponding (under the identification in Proposition~\ref{prop:Wpm-1c-identification}) bicharacteristics meet $\WF'_{\ps}(P - P_0)$ and
$Q_{\ps} \in \Psi_{\ps}^{0,0}(\R^{n+1})$ is microlocally equal to the identity on $\WF'(P - P_0)$. 
Then our characterization of the microlocalized forward and backward Poisson operators is as follows.
\begin{prop}{\cite[Proposition~5.16]{HJ2026-scattering-map}}
\label{prop:microlocalized Poisson are 1c-ps FIOs} 
Let $Q_{\oc}$ and $Q_{\ps}$ be as above, then 
\begin{equation}
    Q_{\ps} \Poipm Q_{\oc} \in I^{-3/4}_{\oc-\ps}(\R^{n+1} \times \R^n, \Lambda_\pm).
\end{equation}
\end{prop} 

The reason for investigating forward and backward Poisson operators is that the scattering map can be constructed out of them directly. More precisely, we have:
\begin{equation} \label{eq:S formula} 
 S = i(2\pi)^n \mathcal{P}_+^* [P,Q_+] \mathcal{P}_-,
\end{equation}
where $Q_+ \in \Psi_{\ps}^{0,0}(\R^{n+1})$ is microlocally equal to the identity on a neighborhood of $\mathcal{R}_+$ with $\WF'_{\ps}(Q_+)$ contained in a slightly enlarged neighborhood of $\Radp$. In particular, $\WF'_{\ps}([P,Q_+])$ is away from both of $\mathcal{R}_\pm$. See \cite[Section~7]{gell2022propagation} (with a correction on an overall sign in \cite[Eq~(1.7)]{HJ2026-scattering-map}).

\section{The 1c-1c analysis and the structure of the scattering map}
\label{sec:1c-1c-calculus}

In this section, we recall the main geometric and analytic ingredients needed for the theory of $\oc-\oc$ Fourier integral operators from \cite[Section~6, Section~7]{HJ2026-scattering-map}. This is used to characterize our scattering map.

\subsection{Geometry of the 1c-1c phase space}

Let $X$ be a manifold with boundary $\partial X$ (in our case, $X=\overline{\R^n}$), the b-double space introduced by Melrose \cite{melrose1993atiyah} is defined to be
\begin{align} \label{eq: defn  b-double space}
X_b^2 : = [ X \times X ; \partial X \times \partial X ],
\end{align}
and the blow down map $X_b^2 \rightarrow X^2$ is denoted $\beta_b$.
Then the b-lifted 1c-1c cotangent bundle is
\begin{align} \label{eq:def-lifted-ococ-bundle}
\ococb:= \beta_b^*(\overline{{}^{\oc}{T^*X}} \times \overline{{}^{\oc}{T^*X}}),
\end{align}
where the right hand side is viewed as pulling back a bundle over $X^2$ to $X_b^2$.
We denote the corresponding projection map $\ococb \to \ocphase \times \ocphase$ by $\tilde{\beta}_b$.
We denote the lift of $\partial X \times X, X \times \partial X, \partial X \times \partial X$ under $\beta_b$ by $\mathrm{lb}$ (`left boundary'), $\mathrm{rb}$ (`right boundary') and $\mathrm{bf}$ (`b-face') respectively. We denote the part of the bundle $\ococb$ lying over $\mathrm{bf}$ by $\ffococ$. 
Let $(x_{\oc,1}, y_1;x_{\oc,2}, y_2)$ be coordinates on $X^2$, then we use the notation
\begin{align*}
x_{\oc} = x_{\oc,1}, \; \sigma = \frac{x_{\oc,1}}{x_{\oc,2}},
\end{align*}
and $\msf{X}:=(x_{\oc},\sigma, y_1, y_2)$ forms a local coordinate system of $X_b^2$ on the region $\{ C^{-1} \leq \sigma \leq C \}$ for a fixed $C$, which is the interesting part for us.
The canonical 1-form on $\ococb$ is given by the sum of the canonical one form lifted from the left and right factors:
\begin{equation}\label{eq:1c1c-1form--1}
 \alpha_{\oc-\oc} = \xioco \frac{d\xoco}{\xoco^3} + \etaoco \frac{d\yoco}{\xoco} + \xioct \frac{d(\xoco/\sigma)}{(\xoco/\sigma)^3} + \etaoct \frac{d\yoct}{(\xoco/\sigma)},
\end{equation}
where we used  $\xoct = \xoco/\sigma$. 
Its differential gives our symplectic form 
\begin{equation}\label{eq:omega 1c1c}
\omega_{\oc-\oc} = d \alpha_{\oc-\oc}. 
\end{equation}
Then the class of Legendre and Lagrangian submanifolds we will use is the following.

\begin{definition}[Admissible 1c-1c Lagrangian submanifold and 1c-1c fibred-Legendre submanifold]  \label{definition:admissible-1c1cLag}
We define an admissible $\oc-\oc$ Lagrangian submanifold of $\ococb$ to be a $2n$-dimensional submanifold $\Uplambda$ that is Lagrangian in the interior (that is, the symplectic form $\omega_{\oc-\oc}$ from \eqref{eq:omega 1c1c} vanishes on it), such that 
\begin{itemize}
\item $\Uplambda$ meets $\ffococ$ transversally, 
\item the differential $d\xioco$ is nonvanishing on $\Uplambda \cap \ffococ$, and 
\item  its closure is disjoint from all other boundary hypersurfaces of $\ococb$ (other than $\ffococ$). 
\end{itemize}
We define a fibered-Legendre submanifold $\SL$ of $\ffococ$ to be the boundary of an admissible 1c-1c Lagrangian submanifold. In other words, there exists $\Uplambda$ as above such that $\SL = \Uplambda \cap \ffococ$. 
\end{definition}

\subsection{The calculus of 1c-1c Fourier integral operators and the scattering map}
\label{subsec:1c-1c-calculus}
In this subsection, we recall the calculus of 1c-1c Fourier integral operators, which is used to characterize our scattering map.

See \cite[Definition~6.4, 6.6]{HJ2026-scattering-map} for details of the parametrization of 1c-1c Lagrangian submanifolds.
From \eqref{eq:S formula}, the 1c-1c Fourier integral operator that we consider will arise from composing a 1c-ps Fourier integral operator and its adjoint.
So if we use the phase function in \eqref{eq:1c-ps-phase-normal-form} in the 1c-ps Fourier integral operators, then the corresponding composition will be the difference of two such functions. Then we have the following normal form for parametrizing 1c-1c Legendre submanifolds (and their Lagrangian extensions):
\begin{align} \label{eq:1c-1c-phase-normal-form}
\Phi_{\oc-\oc}(\msf{K},t, \theta_1) =  \frac{ t(1-\sigma^2) }{\xoco^2} + \frac{\varphi_1(\msf{K},t,\theta_1)}{\xoco} , \quad \msf{K} = (\xoco, \sigma, \yoco, \yoct).
\end{align}


\begin{definition} \label{def:1c-1c-FIO}
Let $\Uplambda$ be an admissible Lagrangian submanifold of $\ococb$ as in Definition~\ref{definition:admissible-1c1cLag}, with boundary $\SL$. 
We define 
$I^{m}_{\oc-\oc}(X_b^2, \SL; \Omega_{\oc-\oc}^{1/2})$, i.e., the space of 1c-1c fibered-Legendre distributions of order $m$, to be the space of operators with Schwartz kernel given (modulo a Schwartz function) by a finite sum of terms of the form 
\begin{align} \label{eq: 1c-1c FIO, local form}
\begin{split}
 & (2\pi)^{-\frac{n+(k_0+k_1-e)}{2}} \Big(\int 
e^{i\Phi_{\oc-\oc}(\msf{K},v,w)} 
 a(\msf{K},v,w)
\\& x_{\oc,1}^{-m-\frac{2k_0+(k_1-e)}{2}+ \frac{n+1}{2}} dvdw \Big) 
|\frac{d\sigma dy_{\oc,2}}{x_{\oc,1}^{n+1}}|^{1/2}|\frac{dx_{\oc,1}dy_{\oc,1}}{x_{\oc,1}^{n+2}}|^{1/2},
\end{split}
\end{align}
where $\Phi_{\oc-\oc}$ is a phase function 
parametrizing $\mathcal{L}$. 
Moreover,  $a \in C^\infty_c([0,\infty)_{x_{\oc}} \times [C^{-1},C]_{\sigma} \times \mathbb{S}^{n-1} \times \mathbb{S}^{n-1} \times \R^{k_0} \times \R^{k_1})$, where $v \in \R^{k_0},w \in \R^{k_1}$. 
\end{definition}

In the same way as the 1c-ps case above, 
over a region on which two phase functions are clean parametrizations, the oscillatory integral above using one of them can be written as an oscillatory integral using the other as the phase function, modulo a Schwartz error. In addition, allowing a clean phase function in \eqref{eq: 1c-1c FIO, local form} instead of using non-degenerate phase functions only does not enlarge the operator class, see \cite[Remark~7.2, Proposition~7.3]{HJ2026-scattering-map}.

In a manner parallel to the 1c‑ps setting, let $S_{\oc-\oc}^{[m]}(\SL)$ be the line bundle that is homogeneous of degree $-m- \frac{n+1}{2}$ in $x_{\oc}$, incorporating a factor for the choice of the boundary defining function and the Maslov bundle and let $\Omega^{1/2}(\SL)$ be the half-density bundle on $\SL$. See \cite[Section~7.2]{HJ2026final} for details. 
We can view the principal symbol map as
\begin{align} \label{eq: invariant, symbol map, 1c-1c}
\begin{split}
\sigma^m_{\oc-\oc}: \quad  I^m_{\oc-\oc}(X_b^2,\Uplambda;\Omega_{\oc-\oc}^{1/2}) 
\rightarrow 
 C^\infty(\SL \cap \partial X_b^2; 
\Omega^{1/2}(\SL) \otimes S_{\oc-\oc}^{[m]}(\SL)).
\end{split}
\end{align}
The fact that this principal symbol map captures the leading order singularity can be summarized in the following short exact sequence from \cite[Proposition~7.4]{HJ2026-scattering-map}:
\begin{align} \label{eq:1c-1c-FIO-short-exact-sequence}
\begin{split}
0  \rightarrow I^{m-1}_{\oc-\oc}(X_b^2,\SL)
\rightarrow I^{m}_{\oc-\oc}(X_b^2,\SL)
\xrightarrow{\sigma^m_{\oc-\oc}} C^\infty(\SL \cap \partial X_b^2; 
\Omega^{1/2}(\SL) \otimes S_{\oc-\oc}^{[m]}(\SL)) \rightarrow 0.
\end{split}
\end{align}

Next we recall results concerning the composition of $\oc-\oc$ Fourier integral operators.
Let $\mathcal{L}_i = C_i' \cap \ffococ$, $i=1,2$ be admissible 1c-1c Legendre submanifolds with $C_i'$ being the corresponding Lagrangian submanifolds.
Suppose $C_1,C_2$ are lifts of canonical relations in $\leftidx{^{\oc}}{T^*X} \times \leftidx{^{\oc}}{T^*X}$ to $\ococb$, and $C_2 \times C_1$ intersects the lift to $\ococb \times \ococb$ of the diagonal in the second and third components of
\begin{align*}
\leftidx{^{\oc}}{T^*X} \times \leftidx{^{\oc}}{T^*X} \times \leftidx{^{\oc}}{T^*X} \times \leftidx{^{\oc}}{T^*X}
\end{align*}
transversally. 
Recall from \cite[Theorem~4.2.2]{FIO1} 
that there is a natural bilinear map giving the product of density bundles:
\begin{align}  \label{eq: symbol product, bundle maps}
\begin{split}
& \Omega^{1/2}(\mathcal{L}_2) \otimes S^{[m_2]}(\mathcal{L}_2)
 \times \Omega^{1/2}(\mathcal{L}_1) \otimes S^{[m_1]}(\mathcal{L}_1) 
 \\ & \rightarrow  \Omega^{1/2}((C_2 \circ C_1)' \cap \ffococ) \otimes S^{[m_2+m_1]}( (C_2 \circ C_1) ' \cap \ffococ),
\end{split}
\end{align}
where $(C_2 \circ C_1)$ denotes the composition of canonical relations.
Denoting this bilinear map by `$\times$', we have the following composition law in the calculus of $\oc-\oc$ Fourier integral operators.
\begin{prop}{Adapted version of \cite[Proposition~7.5]{HJ2026-scattering-map}}
\label{prop: 1c-1c transversal composition}
Suppose $C_2 \times C_1$ satisfies the transversal intersection condition above, and $A_1 \in I^{m_1}_{\oc-\oc}(X^2_b,C_1'), A_2 \in I^{m_2}_{\oc-\oc}(X^2_b,C_2')$, then we have
\begin{align*}
A_2A_1 \in I^{m_1+m_2}_{\oc-\oc}(X_b^2, (C_2 \circ C_1)').
\end{align*}
And when they have $a_1,a_2$ as their principal symbols respectively, then $A_2A_1$ has principal symbol
\begin{align*}
a_2 \times a_1.
\end{align*}
In addition, when $C_1,C_2$ are both graphs of symplectomorphisms, if $A_1,A_2$ are elliptic (in the sense defined after \eqref{eq:1c-1c-FIO-short-exact-sequence}) at $q_1',q_2'$ such that $(q_2,q_1)$ is sent to $q \in C_2 \circ C_1$, then $A_2A_1$ is elliptic at $q'$.
\end{prop}

Now we define the classical scattering map $\Cl$, whose graph will have a natural correspondence to bicharacteristic lines and will be used to define our 1c-1c Lagrangian submanifold.
Given $q \in \overline{{}^{\oc}T^* \SR_-}$, under the identification with $W_-$ in Proposition~\ref{prop:Wpm-1c-identification}, there is a unique bicharacteristic that tends to it in the backward direction.
This bicharacteristic will tend to a point $q' \in \overline{{}^{\oc}T^*\SR_+}$ in the forward direction, again after being identified with $W_+$ as above.
Then we define
\begin{equation} \label{eq:classical-sc-map-def}
    \Cl(q) = q'.
\end{equation}
This is a smooth map since the flow of $H_p^{2,0}$ reaches $W_\pm$ in finite time transversally.

The 1c-1c Lagrangian submanifold that we are going to use for our scattering map is the following.
Let $\tilde{\beta}_b$ be the projection map $\ococb \to \overline{{}^{\oc}{T^*X}} \times \overline{{}^{\oc}{T^*X}}$, and use $\mathrm{Gr}(\Cl)'$ to denote the twisted (i.e., with the sign of the frequency variable of the second component flipped) graph of $\Cl$ in $\overline{{}^{\oc}{T^*X}} \times \overline{{}^{\oc}{T^*X}}$, then the admissible 1c-1c Lagrangian submanifold that we will use is $\tilde{\beta}_b^*(\mathrm{Gr}(\Cl)')$, which is defined to be the closure of the preimage of the interior part of $\mathrm{Gr}(\Cl)'$.
Then the characterization of the scattering map using our calculus of 1c-1c Fourier integral operators is the following.
\begin{thm}{\cite[Theorem~1.1]{HJ2026-scattering-map}}
\label{thm: sc-map-HJ}
The scattering map $S$ is an elliptic 1-cusp Fourier integral operator of order zero, with canonical relation the graph of the classical scattering map. 
\begin{align}\label{eq:S 1c-1c FIO}
S \in I^0_{\oc-\oc}(X_b^2,\tilde{\beta}_b^*(\mathrm{Gr}(\Cl)')), \quad X = \overline{\R^n}. 
\end{align}
The scattering map $S$ acts as the identity microlocally on functions (asymptotic data) supported in a compact subset of $\R^n$, or are supported microlocally near frequency-infinity in the 1-cusp sense. 
\end{thm}

We conclude this section by summarizing how $\Cl$ and bicharacteristic lines of $P$ relate to geodesics of $g(t)$.
In the correspondence of Proposition~\ref{prop:Wpm-1c-identification}, the position variable in $\overline{{}^{\oc}T^*\R^n}$ corresponds to the frequency variable in $\overline{{}^{\ps}T^*\R^{n+1}}$. 
As discussed after \eqref{eq:CharP-def}, the only bicharacteristics meeting the metric and potential perturbations are those ones with infinite frequency.
Under the identification in Proposition~\ref{prop:Wpm-1c-identification}, these bicharacteristics in fiber infinity of $\overline{{}^{\ps}T^*\R^{n+1}}$ are identified with points in 
$\overline{{}^{\oc}T^*_{\partial \overline{\R^n}}\R^n}$.
Then in terms of $\mathrm{Gr}(\Cl)$ above, they will be canonically identified with points in
\begin{equation}
  \tilde{\beta}_b^*(\mathrm{Gr}(\Cl)') \cap  \ffococ.
\end{equation}
The part of $\Char(P)$ at fiber infinity over $B_R(0)$, after projecting out the frequency dual to $t$ and identifying the fiber infinity as a sphere, can be identified with
\begin{equation} \label{eq:parametrized-copehre-bundle-def}
    \mk{B}_g = [-T,T] \times S^*_gB_R(0),
\end{equation}
which is a family of the sphere bundle $S^*_g\R^n$ parametrized by time $t \in [-T,T]$.
More concretely, we introduce:
\begin{equation}  \label{eq:iota-def}
\iota: \;  \quad (t,z,v) \in \mk{B}_g   \to  (t,z,0,-1,v) \in \Char(P),
\end{equation}
where the coordinates are as in \eqref{eq:ps-coordinates},
with $\hat{\zeta}$-part parametrized by $v$, the fiber part of $S^*_gB_R(0)$.

When we send a point from $\overline{{}^{\oc}T^*_{\partial \overline{\R^n}}\R^n}$ to a bicharacteristic line, $t$ is determined by $t=-\frac{1}{2}\xi_{\oc}$ by \cite[Eq.(3.27)]{HJ2026-scattering-map} and fixed over the entire bicharacteristic line.
Using $\iota$ in \eqref{eq:iota-def}, each bicharacteristic line, after forgetting the (rescaled) frequency dual to $t$ and the $\rho_{\ps}$ component, is just a geodesic (lifted to the cosphere bundle) of $g(t)$. 
In addition, those geodesics entering $\mk{B}_g$, or the image of $\mk{B}_g$ under $\iota$ above, correspond to points in a region of $\overline{{}^{\oc}T^*_{\partial \overline{\R^n}}\R^n}$ on which 1-cusp frequencies are bounded.
So the discussion above also gives the correspondence between these geodesics and points in $\tilde{\beta}_b^*(\mathrm{Gr}(\Cl)') \cap  \ffococ$, which plays the role of `end points' of those geodesics and we summarize it below.

\begin{prop}  \label{prop:classical-SC-geodesic}
Let $\xi_{\oc}$ be the frequency as in \eqref{eq: 1c- canonical form} lifted from the left factor to $\ococb$, then each point in $\tilde{\beta}_b^*(\mathrm{Gr}(\Cl)') \cap  \ffococ$ corresponds to a geodesic of $g(t)$ with $t=-\frac{1}{2}\xi_{\oc}$ or a bicharacteristic line of $P$, so that its left and right projections to $\overline{{}^{\oc}T^*_{\partial \overline{\R^n}}\R^n}$ correspond to the initial and ending point of the geodesic as above.
In addition, those geodesics entering $\mk{B}_g$ correspond to points in $\overline{{}^{\oc}T^*_{\partial \overline{\R^n}}\R^n}$ with 1-cusp frequency in a bounded region.
\end{prop}

Next we consider the composition of two 1c-ps Fourier integral operators,
which is motivated by \eqref{eq:S formula} and will be used to give the expression of the principal symbol of the scattering map.
\begin{prop}{\cite[Theorem~8.3]{HJ2026-scattering-map}} \label{prop: 1c-ps composition to 1c-1c}
Suppose $A_- \in I_{\oc-\ps}^{m_-}(\R^{n+1} \times \R^n,\Lambda_-)$, $A_+ \in I_{\oc-\ps}^{m_+}(\R^{n+1} \times \R^n,\Lambda_+)$, then we have
\begin{align}
A_+^*A_- \in I_{\oc-\oc}^{m_++m_-+\frac{1}{2}}(X_b^2,\beta_b^*(\mathrm{Gr}(\Cl)')).
\end{align}
The principal symbol of $A_+^*A_-$ at $(q-, q_+)$, where $q_- \in W_-$ and $q_+ = \Cl(q_-)$, is 
\begin{align} \label{eq: principal symbol, 1c-ps composition to 1c-1c}
\int_{-\infty}^\infty \overline{\sigma_{1c-ps}^{m_+}(A_+)} \sigma_{1c-ps}^{m_-}(A_-) (q_-, \gamma_{q_-}(s)) ds ,
\end{align}
where the integral is a global section of 
\begin{equation}
\Omega^{1/2}(\beta_b^*(\mathrm{Gr}(\Cl))) \otimes S_{\oc-\oc}^{[m_-+m_++\frac{1}{2}]}(\beta_b^*(\mathrm{Gr}(\Cl))).    
\end{equation}
\end{prop}

\section{Determining the potential by the Poisson operator}
\label{sec:Poisson-determine-potential}

In this section, we prove Theorem~\ref{thm:main-Poisson-intro}, which shows how the potential can be determined by the sub-leading part of the Poisson operator. 

\subsection{The microlocalized Poisson operators and their principal symbols}
\label{subsec:Poisson-principal}

We derive detailed information about the principal symbol of the forward and backward Poisson operators in this subsection.

Recall that in the parametrix construction in the proof of \cite[Proposition~5.16]{HJ2026-scattering-map}, we showed that the microlocalized forward and backward Poisson operators $Q_{\pm,\ps}\Poipm Q_{\pm,\oc}$ equal to the parametrix $K_\pm$ constructed in \cite[Proposition~5.15]{HJ2026-scattering-map}. Here $K_\pm$ is an asymptotic sum
\begin{align} \label{eq:Kpm-sum}
K_\pm = \sum_{j=0}^\infty K_{\pm,j}, 
\quad K_{\pm,j}  \in I^{-j -\frac{3}{4} }_{\oc-\ps}(\R^{n+1} \times \R^n, \Lambda_\pm),
\end{align}
such that
\begin{align} \label{eq: parametrix property, up to N}
\Big( P \sum_{j=0}^N K_{\pm,j} - Q_1\mathcal{P}_0 Q_{\oc} \Big) \in I^{-\frac{3}{4}-N}_{\oc-\ps}(\R^{n+1} \times \R^n, \Lambda_\pm).
\end{align} 
We denote the principal symbol of $K_{\pm,j}$ by $a_{\pm,j}$ and we will investigate $a_{\pm,0}$ in this subsection and $a_{\pm,1}$ in the next subsection.

To obtain the detailed information beyond existence, we need more refined information about the transport equation from which we constructed $a_{\pm,j}$.
First, we show that the subprincipal symbol 
\begin{equation}
    p_{\sub}  = r - (2i)^{-1} \sum_k \frac{\partial^2p_{\full}}{\partial z_k \partial \zeta_k} \mod \; S_{\ps}^{0,0}
\end{equation}
used in the parametrix construction can be chosen to be purely imaginary, where $r,p_{\full}$ are recalled below. To this end, we consider the expression of the left full symbol of $P$. 
For the rest of this section, repeated indices are summed.
Unravelling $\Delta_g$ as
\begin{equation}
\Delta_g = \sum_{i,j} |g|^{-1/2}D_{z_i}
\big(|g|^{1/2}g^{ij}D_{z_j}\big) = g^{ij}D_iD_j+|g|^{-1/2}D_i(g^{ij}|g|^{1/2})D_j,
\end{equation}
we know the left full symbol of $P = D_t+\Delta_g+V$ is
\begin{align} \label{eq:full-symbol}
p_{\mathrm{full}}(t,z,\tau,\zeta)
&= \tau+g^{ij}(t,z)\zeta_i\zeta_j
   - i\,b^j(t,z)\zeta_j
   + V(t,z),
\end{align}
where
\begin{align} \label{eq:def-bj}
b^j = |g|^{-1/2}\partial_{z_i}    \big(|g|^{1/2}g^{ij}\big).
\end{align}
One can derive the desired property without using the concrete expression of $b^j$, but we use this explicit formula below since the resulting $p_{\sub}$ is quite geometric. We have 
\begin{equation}
\frac{\partial^2 p_{\full}}{\partial z_j \partial \zeta_j } =
2  \partial_{z_j}g^{ij}\zeta_i-i \partial_{z_j}b^j .
\end{equation}
Recalling the definition of $r,p_{\sub}$ in \cite[Section~5.3]{HJ2026-scattering-map}, the homogeneous principal symbol is 
\begin{align}
    p_{\mathrm{hom}} = \tau + g^{ij}\zeta_i\zeta_j.
\end{align}
Then the remainder part of the symbol is 
\begin{align}
    r=p_{\mathrm{full}} - p_{\hom} = - i\,b^j(t,z)\zeta_j
   + V(t,z). 
\end{align}
By the computation above, we have
\begin{align}
\begin{split}
    p_{\sub} & = r - (2i)^{-1} \sum_k \frac{\partial^2p_{\full}}{\partial z_k \partial \zeta_k}
\\ &= -i b^j\zeta_j + V
   - (2i)^{-1} \left( 2(\partial_{z_k}g^{kj})\zeta_j  - i\partial_{z_k}b^k \right) 
   \\ &= -i b^j  \zeta_j
   + i(\partial_{z_k}g^{kj})\zeta_j
   + V + \frac{1}{2}\partial_{z_k}b^k 
   \\ & = -\frac{i}{2}\zeta_j g^{jk}\partial_{z_k}\log|g|
   + V+\frac{1}{2}\partial_{z_k}b^k.
\end{split}
\end{align}
Since $p_{\sub}$ is defined invariantly only modulo $S^{0,0}_{\ps}$, we have
\begin{align}
p_{\mathrm{sub}}
&\equiv -\frac{i}{2}g^{ij}\partial_{z_i}\log |g|\,\zeta_j \quad \mod S^{0,0}_{\mathrm{ps}}.
\end{align}
So we may use $-\frac{i}{2}g^{ij}\partial_{z_i}\log |g|\,\zeta_j$, which is purely imaginary, as the representative of $p_{\sub}$.


We now derive the expression of the principal symbol of the forward and backward Poisson operators $\Poipm$. 
Let $q_\pm \in \ocphase$ be points corresponding to the starting and ending points of $\gamma$ under the identification in Proposition~\ref{prop:classical-SC-geodesic}.  So $q_+ = \Cl(q_-)$.
We write $(q_\pm,\gamma(s))$ as variables on $\Lambda_\pm$, which should be interpreted as lifted to the 1c-ps phase space, rather than in this product form. We consider $q_-$ such that the bicharacteristic line starting from it meets $B_R(0)$. We can choose $Q_{\oc}$ so that $q_{\oc} \gtrsim 1$ on this region. 

Recall from \cite[Eq.(5.41)]{HJ2026-scattering-map}, the principal symbol of the parametrix solves the following transport equation arising from Proposition~\ref{prop: vanishing principal symbol product} along a bicharacteristic line $\gamma(\cdot)$:
\begin{align} \label{eq: transport ODE, a-0-2}
(-i\mathscr{L}_{ H_p^{2,0} } +i A)a_{\pm,0} = F_\pm,
\end{align}
where 
\begin{align}
\begin{split}
F_\pm(s) = \sigma^{1/4}(Q_{\pm,1}\mathcal{P}_0 Q_{\oc})(q_\pm,\gamma(s)), \quad
A(s) = (\frac n2+1) (x_{1c}^{-1}H_p^{2,0}x_{1c})(q_\pm,\gamma(s)) -i p_{\mathrm{sub}},
\end{split}
\end{align}
with the $q_\pm$-dependence abbreviated.

Here $A(s)$ is real, and $a_{-,0}$ satisfies the `initial condition' $a_{-,0}(q_-,\gamma(s)) = 0$ for sufficiently negative $s$ and the `final condition' $a_{+,0}(q_+,\gamma(s))=0$ for sufficiently large $s$. 
In addition, let $\tilde{Q}_{\pm,\ps} \in \Psi_{\ps}^{0,0}(\R^{n+1})$ be such that $\WF'_{\ps}(\tilde{Q}_{-,\ps})$ is only backward in terms of the $H_p^{2,0}$-flow relative to $B_R(0)$ and $\WF'_{\ps}(\tilde{Q}_{+,\ps})$ is only forward relative to $B_R(0)$.
We set
\begin{equation} \label{eq:def-Q1-minus}
   Q_{\pm,1} = \chi[P,\tilde{Q}_{\pm,\ps}] \in \Psi_{\ps}^{1,-\infty},
\end{equation}
where $\chi$ is a smooth cut-off supported in a large ball in $\R^{n+1}$. See also \cite[Eq.(5.36)]{HJ2026-scattering-map}. 
In particular, $Q_{-,1}$ has purely imaginary principal symbol.

Since we are in a region with $q_{\oc} \gtrsim 1$, the supports of $F_\pm$ are non-empty. We use $s_\pm(q_\pm)$ to denote the time when $\gamma(\cdot)$, after fixing a unit-speed parametrization, enters and leaves this support:
\begin{align} \label{eq:def-s-s+}
   s_-(q_-) = \inf \{ s: \gamma(s) \in \supp \, F_-(s) \}, \quad
   s_+(q_+) = \sup \{ s: \gamma(s) \in \supp\, F_+(s) \}.
\end{align}
Then by the choice of $\tilde{Q}_{\pm,\ps}$ above, we know that $\gamma(s_-)$ is a point before $\gamma$ enters $B_R(0)$, which contains the support of all perturbations, and $\gamma(s_+)$ is a point after $\gamma$ leaves $B_R(0)$.




To make the expressions below more concise, we introduce
\begin{equation} \label{eq:def-E}
    E(s',s) = \exp(\int_{s'}^{s} A(s'')d s'').
\end{equation}
Solving the ODE \eqref{eq: transport ODE, a-0-2}, we have the following expression for $a_{-,0}$. 
\begin{lmm} \label{lemma:sign-symbols}
    The principal symbol of $\Poim$ is given by
    \begin{equation} \label{eq:a-0,explicit-1}
a_{-,0}(q_-,\gamma(s)) = i \int_{-\infty}^s E(s',s) F_-(s')ds'.
\end{equation}
Similarly, we have 
\begin{equation} \label{eq:a+0-explicit-1}
a_{+,0}(q_+,\gamma(s)) = -i \int_{s}^\infty
E(s',s) F_+(s')ds'.
\end{equation}
In particular, $a_{\pm,0}$ (hence $K_{\pm,0}$) is independent of the potential.

In addition, we can choose $\tilde{Q}_{\pm,\ps},Q_{\pm,1}$ mentioned above so that $F_\pm$ is purely imaginary and \footnote{Rigorously, $F_\pm$ is a section of the bundle in \eqref{eq: invariant, symbol map, 1c-ps}. Here and below, statements about being real, imaginary or having a sign refer to the scalar factor obtained after fixing a boundary defining function, a trivialization of the density bundle and removing the phase correction introduced by the Maslov bundle.  }
\begin{align} \label{eq:Fpm-definite-sign}
    \pm i F_{\pm} \geq 0,
\end{align}
and the strict inequality holds when $s \in (s_-(q_-),s_+(q_+))$.

For $q_\pm$ as above and $s \in (s_-(q_-),s_+(q_+))$, we have $\quad a_{\pm,0} > 0$, and for $s$ with the projection of $\gamma(s)$ in $\R^{n+1}$ lying in $[-T,T] \times B_R(0)$, we have
\begin{equation} \label{eq:apm0-definite-sign}
    a_{\pm,0}(q_\pm,\gamma(s)) \gtrsim 1.
\end{equation}

\end{lmm}

\begin{proof}
The concrete expressions for $a_{\pm,0}$ follow from solving the ODE \eqref{eq: transport ODE, a-0-2} directly. 

The fact that $F_-(s')$ can be chosen to be purely imaginary follows from our choice of $Q_1$ above: since the principal symbol of $[P,\tilde{Q}_{\pm,\ps}]$ is purely imaginary and $Q_{\pm,1}$ can be obtained from it by multiplying a smooth positive function.
We denote the principal symbol of $\tilde{Q}_{\pm,\ps}$ by 
$\tilde{q}_{\pm,\ps}$.
We can choose $\tilde{q}_{-,\ps}$ (resp. $\tilde{q}_{+,\ps}$) that is decreasing (resp. increasing) along the $H_p^{2,0}$-flow.
The principal symbol of $Q_{\pm,1}=\chi[P,\tilde{Q}_{\pm,\ps}]$ is $i\chi H_p\tilde{q}_{\pm,\ps}$. 
Using Proposition~\ref{prop: PsiDO- 1c-ps FIO composition} and Proposition~\ref{prop: PsiDO- 1c-ps FIO composition-1c-right}, we obtain \eqref{eq:Fpm-definite-sign} since the principal symbol of $\mathcal{P}_0$ is $1$ and $q_{\oc}$, the principal symbol of $Q_{\oc}$, is strictly positive for our $q_-$.
In addition, we can choose $\tilde{q}_{\pm,\ps}$ so that its derivative (along $H_p^{2,0}$) is non-zero on $(s_-(q_-),s_+(q_+))$, so the inequality \eqref{eq:Fpm-definite-sign} is strict when $s \in (s_-(q_-),s_+(q_+))$, which implies \eqref{eq:apm0-definite-sign} in combination with \eqref{eq:a-0,explicit-1} and \eqref{eq:a+0-explicit-1}, since we have integrated those positive functions over an interval with size $\gtrsim 1$.

\end{proof}

\subsection{The sub-principal part of the Poisson operators}

We derive the detailed information about the sub-leading order part of the forward and backward Poisson operators and show that they determine the potential in this subsection.

Let $P_i$ be as in \eqref{eq:def-Pi}. We will use $\bullet^{(i)}$ to denote $\bullet$, which can be an operator or a symbol, constructed in the previous subsection but with $P$ replaced by $P_i$.


Similar to \eqref{eq: transport ODE, a-0-2},
$a_{\pm,1}^{(i)}$ is constructed in \cite[Proposition~5.15]{HJ2026-scattering-map} via solving
\begin{align} \label{eq: transport ODE, a-1}
(-i\mathscr{L}_{ H_p^{2,0} } 
+i A)a_{\pm,1}^{(i)} = \sigma^{-3/4}(P_iK_{\pm,0}-Q_{\pm,1}\mathcal{P}_0 Q_{\oc}), \quad i = 1,2.
\end{align}
So the transport equation for $a^{(1)}_{\pm,1}-a^{(2)}_{\pm,1}$ is
\begin{align} \label{eq: transport ODE, b-1}
(-i\mathscr{L}_{ H_p^{2,0} } 
+i A)(a_{\pm,1}^{(1)}-a_{\pm,1}^{(2)}) = \sigma^{-3/4}((V_1-V_2)K_{\pm,0})=(V_1-V_2)a_{\pm,0},
\end{align}
with the initial condition $\big( a_{\pm,1}^{(1)}-a_{\pm,1}^{(2)}\big)(q_\pm,\gamma(\pm \infty))=0$. 
Here we view $V_1-V_2$ as a 0-th order symbol, lifted to the phase space.
Setting 
\begin{equation}
   b_{\pm,1}=a_{\pm,1}^{(1)}-a_{\pm,1}^{(2)},
\end{equation}
and solving \eqref{eq: transport ODE, b-1} for $b_{-,1}$, we obtain
\begin{equation}
b_{-,1}(q_-,\gamma(s))= i\int_{-\infty}^{s}
E(s',s) \big(V_1-V_2\big)(\gamma(s')) a_{-,0}(q_-,\gamma(s')) ds' .
\end{equation}

Substituting the formula for $a_{-,0}$, we obtain
\begin{equation} \label{eq:b-1-1}
\begin{aligned}
b_{-,1}(q_-,\gamma(s))
= & - \int_{-\infty}^{s}
E(s',s) \big(V_1-V_2\big)(\gamma(s')) \int_{-\infty}^{s'}
E(\rho,s')F_-(\rho)\,d\rho ds' .
\end{aligned}
\end{equation}
Since $E(s',s)E(\rho,s')=E(\rho,s)$, 
we can simplify \eqref{eq:b-1-1} to be
\begin{equation} \label{eq:b-1-formula}
b_{-,1}(q_-,\gamma(s)) = - \int_{-\infty}^{s}
\int_{-\infty}^{s'} E(\rho,s) \big(V_1-V_2\big)(\gamma(s')) F_-(\rho) d\rho ds' .
\end{equation}

Similarly, we solve \eqref{eq: transport ODE, b-1} for $b_{+,1}$ backwardly starting from $s=+\infty$ with $0$ as the initial value to obtain:
\begin{align}
  \begin{split}
  \label{eq:b+1-formula}
b_{+,1}(q_+,\gamma(s)) =&  -i \int_{s}^{+\infty} E(s',s)(V_1-V_2)(\gamma(s'))a_{+,0}(q_+,\gamma(s'))ds'
\\ = & - \int_{s}^{+\infty} \int_{s'}^\infty E(\rho,s)(V_1-V_2)(\gamma(s'))F_+(\rho)d\rho ds' .      
  \end{split}  
\end{align}


We next give a characterization of the mapping property and the compactness of 1c-ps FIOs.

\begin{prop} \label{prop:boundedness-compactness-1cps}
Let $A \in I_{\oc-\ps}^{-\frac{1}{4}-m}(\R^{n+1}\times\R^n,\Lambda_\pm)$. Then $A$ is bounded from $L^2(\R^n)$ to $H_{\ps}^{m,r}(\R^{n+1})$ for any fixed $r \in \R$.

If the principal symbol of $A \in I_{\oc-\ps}^{-\frac{1}{4}-m}(\R^{n+1}\times\R^n,\Lambda_\pm)$ is not identically zero, then $A$ is not a compact operator from $L^2(\R^n)$ to $H_{\ps}^{m,r}(\R^{n+1})$.
\end{prop}

\begin{proof}
Since $A$ has compact support in the $\R^{n+1}$-variable modulo a term in $\mathcal{S}(\R^{n+1} \times \R^n)$, the order $r$ is immaterial and we will take $r=0$ below.

Now we further reduce to the case with $m=0$. Taking $Q = \Op(\la \zeta \ra^m) \in \Psi_{\ps}^{m,-\infty}(\R^{n+1})$, which is bounded and invertible from $H_{\ps}^{m,r}$ to $H_{\ps}^{0,r}$, and $QA \in I_{\oc-\ps}^{-1/4}(\R^{n+1}\times \R^n, \Lagps_\pm)$ by  Proposition~\ref{prop: PsiDO- 1c-ps FIO composition}.
Then $A$ is a bounded (resp. compact) operator from $L^2(\R^n)$ to $H^{m,r}_{\ps}(\R^{n+1})$ if and only if $QA \in I_{\oc-\ps}^{-1/4}(\R^{n+1}\times \R^n, \Lagps_\pm)$ is a bounded (resp. compact) operator from $L^2(\R^n)$ to $H_{\ps}^{0,r}(\R^{n+1})$. So we are reduced to prove the case $m=0$.

Consider the boundedness first. We show that $A: L^2(\R^n) \to L^2(\R^{n+1})$ and $A^*: L^2(\R^{n+1}) \to L^2(\R^n)$ are bounded by a $TT^*$-argument similar to \cite[Section~4.3]{FIOII}. 
Concretely, by the same proof as in the proof of \cite[Theorem~8.3]{HJ2026-scattering-map}, we know 
\begin{align}
    A^*A \in I_{\oc-\oc}^{0}(\R^n\times\R^n, ( (\Lagps_\pm^*)' \circ \Lagps' )'),
\end{align}
where $\circ$ denotes the composition of canonical relations, $\Lagps_\pm^*$ is obtained by switching the left and right variables of $\Lagps_\pm$ and the sign of all frequencies, and $'$ means switching the sign of frequencies on the right variable, which sends a Lagrangian submanifold to the corresponding canonical relation. See \cite[Section~8.1]{HJ2026-scattering-map} for the discussion about details of the composition involving our extra blow-ups. 
A direct computation shows that $( (\Lagps_\pm^*)' \circ \Lagps' )' = {}^{\oc}N^*\Delta_b$, which is the conormal bundle of the lifted diagonal in $\ococb$.
 
On the other hand, the class $I^{0}_{\oc-\oc}(X^2_b;{}^{\oc}N^*\Delta_b)$ is precisely the class of $\oc$-pseudodifferential operators
\begin{align*}
\Psi_{\oc}^{-\infty,0}(\R^n),
\end{align*}
defined in \cite[Section~2.3]{zachos2022inverting} (see also \cite[Section~2.1]{jia2022tensorial}), where $-\infty$ is the differential order and $0$ is the decay order.
In particular, $A^*A$ acts as a bounded map
\begin{equation}
   L^2(\R^n) \to L^2(\R^n).
\end{equation}
So the boundedness of $A,A^*$ follows.

Now we prove the part concerning the compactness. Suppose the claim is not true, i.e., $A: L^2(\R^n) \to L^2(\R^{n+1})$ is compact, then $A^*A: L^2(\R^n) \to L^2(\R^n)$ is compact. However, suppose $\sigma^{-1/4}_{\oc-\ps}(A)$ is non-vanishing at a point in $\partial \Lagps_\pm$ and denote its projection to $\overline{{}^{\oc}T^*\R^n} $ by $q_\pm$, then by \eqref{eq: principal symbol, 1c-ps composition to 1c-1c}, we know that the principal symbol of $A^*A$ is strictly positive at (the lift of) $(q_-,q_+)$, which is in $\partial ({}^{\oc}N^*\Delta_b)$, since it is an integral of a non-negative quantity along a geodesic, and it is positive at a certain point.
Then by the second part of \cite[Lemma~9.7]{jia2026metric-inverse}, we obtain a contradiction, which completes the proof.

\end{proof}

We are ready to prove Theorem~\ref{thm:main-Poisson-intro} now.

\begin{proof}
As mentioned in Remark~\ref{remark:sharpness-main-2},  $\Poi_{\pm}^{(1)}-\Poi_{\pm}^{(2)} \in I_{\oc-\ps}^{-7/4}(\R^{n+1} \times \R^n, \Lagps_\pm)$ by \cite[Proposition~A.2]{jia2026metric-inverse}.
Now suppose $\Poi_{\pm}^{(1)}-\Poi_{\pm}^{(2)}$ is compact from $L^2(\R^n)$ to $H_{\ps}^{3/2,r}(\R^{n+1})$ for a fixed $r$. Let $Q_{\pm,\ps},Q_{\pm,\oc}$ be microlocalizers as in Section~\ref{subsec:Poisson-principal}, then $Q_{\pm,\ps}(\Poi_{\pm}^{(1)}-\Poi_{\pm}^{(2)})Q_{\pm,\oc}$ is compact from $L^2(\R^n)$ to $H_{\ps}^{3/2,r}(\R^{n+1})$ as well.\footnote{In fact, one can show that $\Poi_{\pm}^{(1)}$ coincide with $\Poi_{\pm}^{(2)}$ outside the region we microlocalized to. }
By the second part of Proposition~\ref{prop:boundedness-compactness-1cps} with $m=\frac{3}{2}$, we know 
\begin{equation}
    \sigma^{-7/4}_{\oc-\ps}( Q_{\pm,\ps} (\Poi_{\pm}^{(1)}-\Poi_{\pm}^{(2)}) Q_{\pm,\oc} ) = 0.
\end{equation}
This means $b_{\pm,1}$ in \eqref{eq:b-1-formula}\eqref{eq:b+1-formula}, which solves the ODE \eqref{eq: transport ODE, b-1} vanishes identically, which means the forcing $(V_1-V_2)(\gamma(s))a_{\pm,0}(q_\pm,\gamma(s))$ has to vanish.
Since for $\gamma(s) \in \supp \; (V_1-V_2)$, we have $s \in (s_-(q_-),s_+(q_+))$, as discussed after \eqref{eq:def-s-s+}, in combination with \eqref{eq:apm0-definite-sign}, we have $V_1-V_2 = 0$, which completes the proof.


\end{proof}

\section{Determining the potential by the scattering map}
\label{sec:sc-map-determine-potential}

We show that the subleading part of the scattering map determines the potential in this section. 
We still use $[-T,T] \times B_R(0)$ to denote a compact cylinder that contains the support of $g-g_0$ and $V_i$.

\subsection{Some preliminaries on the Inverse problems}
\label{sec:preliminaries-inverse-problems}

As mentioned in Section~\ref{subsec:strategy-sketch}, we will show that the weighted X-ray transform of $V_1-V_2$ vanishes and this implies that $V_1-V_2$ vanishes by the result below.

Let $g(t)$ be as in Section~\ref{sec:intro}. 
For fixed $t$, let $\beta$ be a geodesic in $\R^n$ with respect to $g(t)$.
The weighted X-ray transform of a function $f$ supported in $B_R(0)$ with weight $w$ is defined by
\begin{align}
 I_wf(\beta) = \int_\beta f(\beta(s)) w(\beta(s),\dot{\beta}(s)) ds,   
\end{align}
which sends a function supported in $B_R(0)$ to a function on the space of geodesics, which can be viewed as a function on $S^*B_R(0)$.




The answer is affirmative in a lot of cases. We state the result by Uhlmann, Vasy and Zhou that we will use.
As the inverse problems are not the main subject of this paper, we can't give a complete survey on the vast literature on this problem.

\begin{thm}[Adapted from \cite{uhlmann2016inverse}] \label{thm:X-ray}
For $n \geq 3$, let $g(t)$ be a family of metrics that admits a convex function in the sense of Definition~\ref{definition:metric-classes}, and let $w$ be a smooth weight function such that $|w| \gtrsim 1$ on $[-T,T] \times S^*B_R(0)$.
Suppose $f \in C^\infty_c([-T,T] \times B_R(0))$ satisfies
\begin{equation}
    I_wf = 0,
\end{equation}
then $f = 0 $.

\end{thm}


\begin{proof}

The result follows from applying \cite[Corollary]{uhlmann2016inverse} to $B_R(0)$ at each fixed $t$.
Here we have an extra weight. But as pointed out in  \cite[Remark~4.3]{uhlmann2016inverse}, this weight that is lower bounded away from zero does not affect the proof and the injectivity of the X-ray transform continues to hold.

\end{proof}

As we have mentioned after Theorem~\ref{thm:main-scmap-intro}, Theorem~\ref{thm:X-ray} is used only in the last step of our proof: concluding that the potential vanishes from the fact that its X-ray transform vanishes. 
We use this version because one of the major advantages of our calculus of Legendrian distributions is that we can handle the presence of conjugate points. 
One can replace this Theorem by another result concerning the injectivity of the X-ray transform and change the condition on $g(t)$ correspondingly.


\subsection{The proof of Theorem~\ref{thm:main-scmap-intro}}

We prove Theorem~\ref{thm:main-scmap-intro} in this subsection. We give a criterion on the compactness of 1c-1c FIOs on Sobolev spaces first.
\begin{lmm} \label{lemma:1c1c-compactness}
Let $A \in I_{\oc-\oc}^{-m}(\R^n \times \R^n, \Lag)$, then it is compact from $L^2(\R^n)$ to $H_{\oc}^{s,m'}(\R^n)$ for any $m'<m$ and fixed $s$.

Conversely, suppose the principal symbol of $A \in I_{\oc-\oc}^{-m}(\R^n \times \R^n, \Lag)$ is not identically zero, then it is not compact from $L^2(\R^n)$ to $H_{\oc}^{s,m}(\R^n)$ for any fixed $s$.
\end{lmm}

\begin{proof}



In the same way as the proof of Proposition~\ref{prop:boundedness-compactness-1cps} and \cite[Lemma~9.7]{jia2026metric-inverse}, we know
\begin{equation} \label{eq:A-Astar}
    AA^* \in \Psi_{\oc}^{-\infty,-2m}(\R^n),
\end{equation}
which is bounded from $H_{\oc}^{-(s+\epsilon),-m}(\R^n)$ to $H_{\oc}^{s+\epsilon,m}(\R^n) = \big(H_{\oc}^{-(s+\epsilon),-m}(\R^n) \big)^*$, where $\epsilon>0$.
So $A: L^2(\R^n) \to H_{\oc}^{s+\epsilon,m}(\R^n)$, $A^*: H_{\oc}^{-(s+\epsilon),-m}(\R^n) \to L^2(\R^n)$ are bounded.
Since the embedding $H_{\oc}^{s+\epsilon,m}(\R^n) \hookrightarrow H_{\oc}^{s,m'}(\R^n)$ is compact, the first conclusion follows.

For the second part, by \eqref{eq:A-Astar} and Theorem~\ref{prop: 1c-1c transversal composition}, we know $AA^*$ is a PsiDO in $\Psi_{\oc}^{-N,-2m}(\R^n)$ that is elliptic at a certain point $q \in \overline{{}^{\oc}T^*_{\partial \overline{\R^n} }\R^n}$, which is not compact from $H_{\oc}^{-s,-m}(\R^n)$ to $H_{\oc}^{s,m}(\R^n)$.
Combining with the boundedness of $A^*$ above (with $s+\epsilon$ replaced by $s$, since both are arbitrary), we know $A$ is not compact from $L^2(\R^n)$ to $H_{\oc}^{s,m}(\R^n)$ as claimed.
    
\end{proof}

Now we are ready to prove Theorem~\ref{thm:main-scmap-intro}. 
\begin{proof}[Proof of Theorem~\ref{thm:main-scmap-intro}]

First recall the formula constructing the scattering map from the Poisson operator used in the proof of \cite[Proposition~9.2]{HJ2026-scattering-map}:
\begin{align} \label{eq:SQ1c-1}
 S_i Q_{\oc} =  i(2\pi)^n \tilde Q_{\oc} (\mathcal{P}_+^{(i)})^* \tilde Q_{\ps} Q_{\ps} [P_i,Q_+] \mathcal{P}_-^{(i)} Q_{\oc} \text{ mod } \mathcal{S}(\R^{2n}). 
\end{align}
Here $Q_+ \in \Psi_{\ps}^{0,0}(\R^{n+1})$ satisfies: the projection of $\WF'_{\ps}([P_i,Q_+])$ to $\R^{n+1}$ is contained in $B_R(0)$. Along each bicharacteristic line, $Q_+$ is microlocally zero before entering (the phase space over) $B_R(0)$, and microlocally identity after leaving $B_R(0)$, and its left full symbol is monotonically increasing along the $H_p^{2,0}$-flow.

For $i=1,2$, let $K_{\pm}^{(i)} = \sum_{j=0}^\infty K_{\pm,j}^{(i)}$ be the parametrix in \eqref{eq:Kpm-sum} constructed for $P_i$. 
By \eqref{eq:SQ1c-1}, we have 
\begin{align} \label{eq:SQ1c-2}
S_iQ_{\oc} = i(2\pi)^n \sum_{j+j' \leq N}  (K_{+,j}^{(i)})^*  [P_i,Q_+] K_{-,j'}^{(i)}  \text{ mod } I_{\oc-\oc}^{-N-1}(\R^{n} \times \R^n, \beta_b^*(\mathrm{Gr}(\Cl)') ), 
\end{align}
where we absorbed microlocalizers into parametrices and they still satisfy properties in Section~\ref{sec:Poisson-determine-potential}.
 
Notice that the potential $V_i$ is two orders lower (in terms of the $\ps$-differential order, which in turn adds to the order of $\oc-\ps$ FIOs by Proposition~\ref{prop: PsiDO- 1c-ps FIO composition}) than the leading part $P_g=D_t+\Delta_g$, so we know that we can replace $P_i$ in the commutator by $P_g$ for \eqref{eq:SQ1c-2} with $N=1$:
 \begin{align} \label{eq:SQ1c-2'}
S_i Q_{\oc} = i(2\pi)^n \sum_{j+j' \leq 1} (K_{+,j}^{(i)})^* [P_g,Q_+] K_{-,j'}^{(i)} \text{ mod } I_{\oc-\oc}^{-2}(\R^n \times \R^n, \beta_b^*(\mathrm{Gr}(\Cl)') ). 
 \end{align}

Consequently, we know the sub-leading part of $S_iQ_{\oc}$ has three terms, the term in \eqref{eq:SQ1c-2} with $j=0,j'=1$, $j=1,j'=0$, or the sub-leading part of the term $j=0,j'=0$.

For the last term, we recall the proof of \cite[Theorem~8.3]{HJ2026-scattering-map}. 
This term is introduced when we conduct the stationary phase to the composition $(K_{+,0}^{(i)})^* [P_g,Q_+] K_{-,0}^{(i)} Q_{\oc}$.
As mentioned after \eqref{eq:a+0-explicit-1},  $K_{\pm,0}^{(i)}$ is independent of the potential, hence so is this term. 
Consequently, the contribution from this term will be cancelled in $(S_1-S_2)Q_{\oc}$.



Using the formula for the principal symbol in Proposition~\ref{prop: 1c-ps composition to 1c-1c}, 
we know that the principal symbol of $(S_1-S_2)Q_{\oc}$, which comes from the sum of terms in \eqref{eq:SQ1c-2'} with $(j,j')=(0,1)$ or $(1,0)$, is
\begin{equation} \label{eq:subleading-S1-S2-1}
\begin{aligned}
\sigma^{-1}_{\oc-\oc}\bigl( (S_1-S_2)Q_{\oc}\bigr) \bigl(q_+,q_-\bigr)  
= & (2\pi)^n \int_{\mathbb R} \bigl(H_p^{2,0}\mk{q}_+\bigr)(\gamma(s))
\times \Big( \overline{b_{+,1}(q_+,\gamma(s))}a_{-,0}(q_-,\gamma(s)) 
\\ & + a_{+,0}(q_+,\gamma(s))b_{-,1}(q_-,\gamma(s)) \Big)  ds ,
\end{aligned}
\end{equation}
where $\mk{q}_+$ is the principal symbol of $Q_+$, $H_p^{2,0}\mk{q}_+$ is the principal symbol of $i^{-1}[P_g,Q_+]$, and we used the fact that $a_{+,0}$ is real.

Using \eqref{eq:a-0,explicit-1}\eqref{eq:a+0-explicit-1}\eqref{eq:b-1-formula}\eqref{eq:b+1-formula}, after changing the order of integration, \eqref{eq:subleading-S1-S2-1} can be rewritten as 
\begin{equation}
\begin{aligned}
&\sigma_{\oc-\oc}^{-1}\bigl((S_1-S_2)Q_{\oc}\bigr)
(q_+,q_-)
\\ = & (2\pi)^n  \Big(
\int_{\mathbb R}
\bigl(V_1-V_2\bigr)(\gamma(s'))
\biggl(
\int_{-\infty}^{s'}
\bigl(H_p^{2,0}\mk{q}_+\bigr)(\gamma(s))
a_{-,0}(q_-,\gamma(s)) 
\big( \int_{s'}^{+\infty} E(\rho,s)F_+(\rho)\,d\rho \big) ds \biggr) d s'
\\ &\quad - \int_{\mathbb R}
\bigl(V_1-V_2\bigr)(\gamma(s'))
\biggl( \int_{s'}^{+\infty} \bigl(H_p^{2,0}\mk{q}_+\bigr)(\gamma(s)) a_{+,0}(q_+,\gamma(s)) \big( \int_{-\infty}^{s'} E(\rho,s)F_-(\rho)\,d\rho \big) ds \biggr) d s' . \Big)
\\ = & -(2\pi)^ni \Big( \int_{\mathbb R}
\bigl(V_1-V_2\bigr)(\gamma(s')) \biggl( \int_{-\infty}^{s'} \bigl(H_p^{2,0}\mk{q}_+\bigr)(\gamma(s)) a_{-,0}(q_-,\gamma(s)) 
\big( \int_{s'}^{+\infty} E(\rho,s) iF_+(\rho)\,d\rho  \big) d s
\\ &\quad - \int_{s'}^{+\infty} \bigl(H_p^{2,0}\mk{q}_+\bigr)(\gamma(s)) a_{+,0}(q_+,\gamma(s)) \big( \int_{-\infty}^{s'} E(\rho,s) iF_-(\rho)\,d\rho \big) ds \biggr) d s' \Big),
\end{aligned}
\end{equation}
where we used the fact that $F_\pm$ is purely imaginary and other factors are real for complex conjugations.

Next we prove that the weight is uniformly away from 0 on the support of $V_1-V_2$.
For $s'$ with $\gamma(s') \in \supp\; (V_1-V_2)$, we have $s' \in (s_-(q_-),s_+(q_+))$.
By Lemma~\ref{lemma:sign-symbols}, we know 
\begin{equation}
  \int_{s'}^{+\infty} E(\rho,s) iF_+(\rho)\,d\rho>0, \quad    \int_{-\infty}^{s'} E(\rho,s) iF_-(\rho)\,d\rho<0.
\end{equation}
By the property of $Q_+$ stated after \eqref{eq:SQ1c-1}, in combination with that $s_-(q_-)$ (resp. $s_+(q_+)$) is smaller (resp. larger) than the times when $\gamma(\cdot)$ enters (resp. escapes) $B_R(0)$ as discussed after \eqref{eq:def-s-s+}, we know $H_p^{2,0}\mk{q}_+$ is a non-negative function supported in $(s_-(q_-),s_+(q_+))$ that integrates to be $1$ along each bicharacteristic line.
For $s \in (s_-(q_-),s_+(q_+))$ we have $a_{\pm,0} \gtrsim 1$ by \eqref{eq:apm0-definite-sign}.
In sum, the absolute value of the weight for any $\gamma(s')$ is $\gtrsim 1$.

By the identification \eqref{eq:iota-def} between bicharacteristic lines with geodesics of $g$ at fixed time, this is a weighted X-ray transform of $V_1-V_2$.
By Lemma~\ref{lemma:1c1c-compactness} with $m=1$, we know that when the compactness condition in Theorem~\ref{thm:main-scmap-intro} holds, this weighted X-ray transform of $V_1-V_2$ vanishes. 
Finally, invoking Theorem~\ref{thm:X-ray} completes the proof of Theorem~\ref{thm:main-scmap-intro}.
\end{proof}

\begin{remark} \label{remark: complex-potential}
    Using the same proof above, if $V_i$ are allowed to take complex values, then $\sigma_{\oc-\oc}^{-1}\bigl((S_1-S_2)$ is still a weighted X-ray transform of $V_1-V_2$.
Since the sign switches for the imaginary part contributed from $\overline{b_{+,1}}a_{-,0}$, the weight is not necessarily lower bounded away from zero now and one can't determine the imaginary part of $V_1-V_2$.
However, the weighted X-ray transform for the real part of $V_1-V_2$ remains the same as above, which has a weight bounded away from $0$, and this argument continues to apply to determine the real part of $V_1-V_2$.
\end{remark}

\bibliographystyle{plain}
\bibliography{bib_sc_map_inverse}

\end{document}